\documentclass{article}

\usepackage{xspace}
\usepackage{marvosym}
\usepackage{microtype}
\microtypesetup{tracking,kerning,spacing}
\microtypecontext{spacing=nonfrench}

\usepackage{hyperref}
\usepackage{amsfonts,amsmath,amssymb,amsthm}
\usepackage{graphicx}
\usepackage{tikz}
\usepackage{pgfplots}
\usepackage{caption}
\usepackage[labelfont=rm]{subcaption}

\usepackage{dsfont}
\usepackage{wasysym}

\usepackage{array}
\usepackage{float}
\usepackage{multirow}
\usepackage{stackrel}

\newtheorem{theorem}{Theorem}[section]
\newtheorem{lemma}[theorem]{Lemma}
\newtheorem{proposition}[theorem]{Proposition}
\newtheorem{corollary}[theorem]{Corollary}
\newcounter{mainTheorem}

\newtheorem{maintheorem}[mainTheorem]{Theorem}
\theoremstyle{definition}
\newtheorem{definition}[theorem]{Definition}

\theoremstyle{remark}
\newtheorem{remark}[theorem]{Remark}

\newtheorem{question}[theorem]{Question}
\newtheorem{example}[theorem]{Example}

\newtheorem{conjecture}[theorem]{Conjecture}

\usetikzlibrary{positioning}

\tikzset{EdgeStyle/.style = {color=black!60, thick}}

\tikzset{VertexStyle/.style = {
    circle, fill=white, draw=black,
    text           = black,
    inner sep      = 2pt,
    outer sep      = 0pt,
    minimum size   = 10 pt}}

\tikzset{BoxVertex/.style = {rectangle, fill=white, draw=black, text = black, inner sep = 2.5pt, outer sep = 0pt, minimum size = 10pt}}

\newcommand{\bigraphtwofourcoord}[5]{
\coordinate (v1) at (#1,#2+1.5*#3);
\coordinate (v2) at (#1,#2+0.5*#3);
\coordinate (v3) at (#1,#2-0.5*#3);
\coordinate (v4) at (#1,#2-1.5*#3);

\coordinate (w1) at (#1-#5,#2+#4);
\coordinate (w2) at (#1-#5,#2-#4);
}
\newcommand{\bigraphtwofournodes}{
\node[BoxVertex] (1) at (v1){1};
\node[BoxVertex] (2) at (v2){2};
\node[BoxVertex] (3) at (v3){3};
\node[BoxVertex] (4) at (v4){4};

\node[VertexStyle] (11) at (w1){1};
\node[VertexStyle] (22) at (w2){2};
}

\newcommand{\bigraphthreetwocoord}[5]{
\coordinate (v1) at (#1,#2+#3);
\coordinate (v2) at (#1,#2);
\coordinate (v3) at (#1,#2-#3);

\coordinate (w1) at (#1-#5,#2+#4);
\coordinate (w2) at (#1-#5,#2-#4);
}
\newcommand{\bigraphthreetwonodes}{
\node[BoxVertex] (1) at (v1){1};
\node[BoxVertex] (2) at (v2){2};
\node[BoxVertex] (3) at (v3){3};

\node[VertexStyle] (11) at (w1){1};
\node[VertexStyle] (22) at (w2){2};
}

\newcommand{\bigraphthreefourcoord}[5]{
\coordinate (w1) at (#1-#5,#2+#4);
\coordinate (w2) at (#1-#5,#2);
\coordinate (w3) at (#1-#5,#2-#4);

\coordinate (v1) at (#1,#2+1.5*#3);
\coordinate (v2) at (#1,#2+0.5*#3);
\coordinate (v3) at (#1,#2-0.5*#3);
\coordinate (v4) at (#1,#2-1.5*#3);
}
\newcommand{\bigraphthreefournodes}{
\node[VertexStyle] (11) at (w1){1};
\node[VertexStyle] (22) at (w2){2};
\node[VertexStyle] (33) at (w3){3};

\node[BoxVertex] (1) at (v1){1};
\node[BoxVertex] (2) at (v2){2};
\node[BoxVertex] (3) at (v3){3};
\node[BoxVertex] (4) at (v4){4};
}

\newcommand{\bigraphthreefivecoord}[5]{
\coordinate (w1) at (#1-#5,#2+#4);
\coordinate (w2) at (#1-#5,#2);
\coordinate (w3) at (#1-#5,#2-#4);

\coordinate (v1) at (#1,#2+2*#3);
\coordinate (v2) at (#1,#2+#3);
\coordinate (v3) at (#1,#2);
\coordinate (v4) at (#1,#2-#3);
\coordinate (v5) at (#1,#2-2*#3);
}
\newcommand{\bigraphthreefivenodes}{
\node[VertexStyle] (11) at (w1){1};
\node[VertexStyle] (22) at (w2){2};
\node[VertexStyle] (33) at (w3){3};

\node[BoxVertex] (1) at (v1){1};
\node[BoxVertex] (2) at (v2){2};
\node[BoxVertex] (3) at (v3){3};
\node[BoxVertex] (4) at (v4){4};
\node[BoxVertex] (5) at (v5){5};
}

\newcommand{\bigraphtwofivecoord}[5]{
\coordinate (w1) at (#1-#5,#2+#4);
\coordinate (w2) at (#1-#5,#2-#4);

\coordinate (v1) at (#1,#2+2*#3);
\coordinate (v2) at (#1,#2+#3);
\coordinate (v3) at (#1,#2);
\coordinate (v4) at (#1,#2-#3);
\coordinate (v5) at (#1,#2-2*#3);
}
\newcommand{\bigraphtwofivenodes}{
\node[VertexStyle] (11) at (w1){1};
\node[VertexStyle] (22) at (w2){2};

\node[BoxVertex] (1) at (v1){1};
\node[BoxVertex] (2) at (v2){2};
\node[BoxVertex] (3) at (v3){3};
\node[BoxVertex] (4) at (v4){$\bar{1}$};
\node[BoxVertex] (5) at (v5){$\bar{2}$};
}

\newcommand{\bigraphthreesixcoord}[5]{
\coordinate (w1) at (#1-#5,#2+#4);
\coordinate (w2) at (#1-#5,#2);
\coordinate (w3) at (#1-#5,#2-#4);

\coordinate (v1) at (#1,#2+2.5*#3);
\coordinate (v2) at (#1,#2+1.5*#3);
\coordinate (v3) at (#1,#2+0.5*#3);
\coordinate (v4) at (#1,#2-0.5*#3);
\coordinate (v5) at (#1,#2-1.5*#3);
\coordinate (v6) at (#1,#2-2.5*#3);
}
\newcommand{\bigraphthreesixnodes}{
\node[VertexStyle] (11) at (w1){1};
\node[VertexStyle] (22) at (w2){2};
\node[VertexStyle] (33) at (w3){3};

\node[VertexStyle] (1) at (v1){1};
\node[VertexStyle] (2) at (v2){2};
\node[VertexStyle] (3) at (v3){3};
\node[BoxVertex] (4) at (v4){1};
\node[BoxVertex] (5) at (v5){2};
\node[BoxVertex] (6) at (v6){3};
}

\newcommand\transpose[1]{{#1}^{\top}}

\DeclareMathOperator{\supp}{supp}
\DeclareMathOperator{\conv}{conv}
\DeclareMathOperator{\sign}{sign}

\newcommand\RR{{\mathbb R}}

\newcommand\Topes{\mathcal{T}\!\!o}

\newcommand\Bcal{\mathcal{B}}

\newcommand\Ical{\mathcal{I}}
\newcommand\Mcal{\mathcal{M}}

\newcommand\modMcal{\widetilde{\mathcal{M}}}
\newcommand\ground{\mathsf{E}}
\newcommand\rows{\mathsf{R}}
\newcommand\augground{\widetilde{\mathsf{E}}}
\newcommand\omm{\mathfrak{O}}
\DeclareRobustCommand{\rchi}{{\mathpalette\irchi\relax}}
\newcommand{\irchi}[2]{\raisebox{\depth}{$#1\chi$}}
\newcommand{\ssimplex}{\triangle}

\global\long\def\ee{\mathbf{e}}

\global\long\def\supp{\mathrm{supp}}

\global\long\def\conv{\mathrm{conv}}

\global\long\def\orientedmatroid{\mathcal{M}}
\global\long\def\omtwo{\mathcal{N}}

\global\long\def\sign{\mathrm{sign}}

\global\long\def\neighbourhood{\mathcal{N}}
\global\long\def\ground{\mathsf{E}}
\global\long\def\rows{\mathsf{R}}

\global\long\def\Rin{\mathtt{Rin}(3,9)}

\makeatletter
\def\input@path{{./IMG1/}{./}}
\makeatother

\graphicspath{{IMG1/}}

\title{Oriented Matroids from Triangulations\\ of Products of Simplices}

\author{Marcel Celaya \\
  Technische Universit\"at Berlin, Institut f\"ur Mathematik \\
  \and
  Georg Loho \\
  London School of Economics, Department of Mathematics\\
  \and
  Chi Ho Yuen \\
  Brown University, Division of Applied Mathematics\\
}

\begin{document}

% \dedicatory{}

\maketitle

\begin{abstract}
  We introduce a construction of oriented matroids from a triangulation of a product of two simplices.
%  This is the first part of our paper, focusing on the combinatorial aspect of our construction.
  For this, we use the structure of such a triangulation in terms of polyhedral matching fields.
  The oriented matroid is composed of compatible chirotopes on the cells in a matroid subdivision of the hypersimplex, which might be of independent interest.
  In particular, we generalize this using the language of matroids over hyperfields, which gives a new approach to construct matroids over hyperfields.
  A recurring theme in our work is that various tropical constructions can be extended beyond tropicalization with new formulations and proof methods.
\end{abstract}

%% \todo[inline]{Should we change 'linkage covectors' and 'covector graphs' to some other name? Candidates:  'type graphs' -- this would go back to earlier terminology of Develin et al. but very generic.

%%   'position / location graph' -- this would be more geometrically inspired.

%%   GL: Change terminology from covector graph to pd-graph!! Add a footnote explaining that this stands for 'primal-dual'. 
%% }

%\todo[inline]{@MC: Please add changes based on Josephine's minor comments. (See Georg's email vom May 25) @All: Are all issues from Josephine addressed now? }

\section{Introduction}

\subsection{Oriented Matroids and Matching Fields}

An {\em oriented matroid} is a combinatorial object abstracting linear dependence over $\RR$, and can be thought as a matroid with sign data, i.e., with signs attached to its bases that satisfy certain exchange axioms.
The standard example of an oriented matroid is given by the signs of the maximal minors of a real matrix, but not all oriented matroids are {\em realizable}, meaning they arise this way.
Oriented matroids play an important role in discrete and computational geometry as well as optimization, ranging from the study of geometric configurations to linear programming; they also make appearances in algebraic geometry and topology \cite[Chapter 1 \& 2]{BLSWZ:1993}.
%In particular, the {\em Topological Representation Theorem} of Folkman and Lawrence \cite{FolkmanLawrence:1978} states that every oriented matroid can be represented by a {\em pseudosphere arrangement} (a topological generalization of real hyperplane arrangements) and vice versa.
%\todo{YCH: I will ask again if we need to keep the Topo Rep Thm sentense.}

A {\em matching field} is a collection of matchings of the complete bipartite graph $K_{\rows,\ground}\cong K_{d,n}$, one perfect matching between $\rows$ and $\sigma$ for every subset $\sigma\subset\ground$ of size $d$;
here $\rows$ is a set of size $d$ \footnote{The notation stands for any of: ``rows", ``realization coordinates", or ``rank".}, and $\ground$ is a set of size $n\geq d$ which we refer to as the {\em ground set}.
The simplest construction of a matching field is by taking all weight-maximal $(d\times d)$ matchings selected by a generic matrix on the complete bipartite graph, yielding a \emph{coherent} matching field. 
They were introduced by Sturmfels and Zelevinsky in \cite{SturmfelsZelevinsky:1993} to capture the combinatorics of the leading terms of maximal minors: a matching field is essentially given by choosing a term from each maximal minor of a generic matrix.
Along with follow-up works such as \cite{BernsteinZelevinsky:1993, LohoSmith:2020}, it was demonstrated that much of the Gr\"{o}bner theory of maximal minors can be deduced from the purely combinatorial {\em linkage} property.
This is analogous to the exchange property of matroids and leads to a generalization of coherent matching fields. 

Signed tropicalization shows that coherent matching fields induce realizable oriented matroids. 
Sturmfels and Zelevinsky noted this in their paper, thereby illustrating the aforementioned analogy between the linkage property and the general exchange axiom. 
Motivated by their remarks, we study the relation between linkage matching fields and oriented matroids, and show that the linkage property is not enough to guarantee an oriented matroid in Example~\ref{ex:SZ_example}.
Nevertheless, our first main result is that the statement is true for {\em polyhedral} matching fields.
Such matching fields arise from triangulations of the product of two simplices.
It is known that every coherent matching field is polyhedral while every polyhedral matching field is linkage, so our result extends the relation between coherent matching fields and realizable oriented matroids to an appropriate generality.
We elaborate more on why polyhedral matching fields form an interesting and important intermediate class of matching fields in the next subsections, but first we state our main theorem.

By fixing an arbitrary ordering for $\rows$ and $\ground$, every matching in a matching field can be thought as a permutation.
We define the sign of the matching as the sign of the permutation.
\begin{maintheorem}[Theorem~\ref{thm:main}] \label{thm:main+theorem+intro}
Given a polyhedral matching field $(M_\sigma)$ and a full $d\times n$ sign matrix~$A$, the sign map $\rchi:\binom{\ground}{d}\rightarrow\{+,-\}$ defined by
\begin{equation}
  \sigma\mapsto\sign(M_\sigma)\prod_{e\in M_\sigma}A_e \enspace
\end{equation}
is the chirotope of a uniform oriented matroid (Definition~\ref{def:chirotope}).
\end{maintheorem}

Conceptually, our theorem says that instead of taking the signs of maximal minors, we can pick out only one term per determinant (carefully) and still obtain an oriented matroid.
This result indicates that the matchings in a polyhedral matching field take the role of signed bases of an oriented matroid.
We show that this correspondence goes even further by considering other special graphs associated with a polyhedral matching field, such as the \emph{linkage pd-graphs}, which are local unions of matchings, and the \emph{Chow pd-graphs}, which are the minimal transversals of the matchings~\cite{LohoSmith:2020, SturmfelsZelevinsky:1993}.
We describe how these graphs directly yield the signed circuits, signed cocircuits, and more generally covectors of the oriented matroid in Theorem~\ref{thm:main+theorem+intro}.
We also develop a notion of {\em duality} for matching fields that descends to the duality of oriented matroids via Theorem~\ref{thm:main+theorem+intro}.

\subsection{Connections to Complexity Questions} \label{sec:connection+complexity}

Our work adds a new piece to the connection between two major open complexity questions.
On one hand, Smale's 9th problem asks for a strongly polynomial algorithm in linear programming.
Oriented matroids play an important role for this as it is the framework for the simplex method~\cite{Fukuda:1982} which is still a natural candidate for such a strongly polynomial algorithm.
On the other hand, determining the winning states of a mean payoff game is a problem in NP~$\cap$~co-NP~\cite{GKK:1988} but no polynomial time algorithm is known.
The latter problem is also equivalent to deciding feasibility of a tropical linear program~\cite{AkianGaubertGuterman:2012} and it turns out that matching fields are the combinatorial framework for describing tropical linear programming~\cite{Loho:2016}.
From the viewpoint of mean payoff games, the matchings can be considered as partial strategies.
While the tropicalization of the simplex method~\cite{ABGJ-Simplex:A,Benchimol:2014} based on sign patterns already gave a connection between pivoting and strategy iteration, we directly derive the correspondence on the level of oriented matroids.
The interplay between classical and tropical linear programming has already lead to a proof that a wide class of interior point methods can not be strongly polynomial~\cite{AllamigeonBenchimolGaubertJoswig:2018}, and we elaborate further in Section~\ref{sec:conclusion} how our work can contribute to the understanding of this interplay. 

\subsection{Triangulations of  $\ssimplex_{d-1}\times\ssimplex_{n-1}$ and Matroid Subdivisions, with Signs}
Triangulations of a product of two simplices are fundamental in combinatorics and algebraic geometry~\cite{ArdilaBilley:2007,DeLoeraRambauSantos:2010, GelfandKapranovZelevinsky:1994}, and via the connection explained in Section~\ref{sec:connection+complexity} also in computer science, just to name a few references. 
{\em Matroid subdivisions} are fundamental objects emerging from the context of valuated matroids~\cite{DressWenzel:1992}, tropical linear spaces~\cite{SpeyerSturmfels:2004}, and discrete convex analysis~\cite{Murota:2003}. 

Drawing various motivations from the literature, we take a crucial new point of view and directly connect triangulations of products of simplices, matroid subdivisions and oriented matroids.
Besides questions from complexity, we were inspired by the connection between regular subdivisions of products of simplices and tropical convexity established in~\cite{DevelinSturmfels:2004}.
This already lead to the concept of a \emph{tropical oriented matroid} by Ardila and Develin~\cite{ArdilaDevelin:2009}, which is equivalent to the (not-necessarily regular) subdivisions of a product of two simplices~\cite{Horn1,OhYoo:2011}.
Polyhedral matching fields, or more precisely the {\em pointed} version thereof, give yet another equivalent description in the generic case of triangulations \cite{LohoSmith:2020, OhYoo-ME:2013}.

In Section 5 of their seminal paper, Ardila and Develin asked for a connection between oriented matroids and tropical oriented matroids, especially for exploring questions related to non-realizability.
The extraction of a realizable oriented matroid from a regular triangulations of $\ssimplex_{d-1}\times\ssimplex_{n-1}$ with signs is implicit in~\cite{Benchimol:2014}, but with our perspective, we finally manage to formalize such speculated connection in full generality. 
As asymptotically almost every triangulation of the product of two simplices is non-regular~\cite{Santos:2005}, our work allows a vast generalization. 
In particular, we derive Ringel's non-realizable oriented matroid of rank 3 on 9 elements from a non-regular triangulation of $\ssimplex_2 \times \ssimplex_5$ in Section~\ref{sec:ringel}, breaking new ground for the study of realizability of oriented matroids. 

Now we sketch our proof of Theorem~\ref{thm:main+theorem+intro}.
Every triangulation of $\ssimplex_{d-1}\times\ssimplex_{n-1}$ gives rise to a matroid subdivision of the hypersimplex by transversal matroid polytopes~\cite{HerrmannJoswigSpeyer:2014}. 
Considering matroid polytopes allows us to manipulate the sign map $\rchi$ geometrically; in particular, the restriction of $\rchi$ to each subpolytope is a realizable chirotope (Lemma~\ref{lem:TM_real}).
To finish the proof, we establish a general local-to-global principle (Theorem~\ref{thm:local_global}) for oriented matroids to show $\rchi$ is a chirotope from local information.
The interaction between oriented matroids and matroid subdivisions seems to be largely unexplored, although some very recent works on positive Dressians and positive tropical Grassmannians \cite{AHrkaniLamSpradlin:2020, LukowskiParisiWilliams:2020, SpeyerWilliams:2020} share some common ideas with ours.
Even more, Section~\ref{sec:compatibility+chirotopes} provides a starting point for understanding the interplay between sign patterns of chirotopes and compatible matroid subdivisions beyond the positivity condition for positroids. 

\subsection{Matroids over Hyperfields}

In \cite{BakerBowler:2019}, Baker and Bowler introduced the theory of {\em matroids over hyperfields}, which provides a common framework for many ``matroids with extra data'' theories.
The base of the theory is the notion of {\em hyperfields}, which can be thought as ordinary fields with multi-valued addition.
Besides unifying parallel notions and propositions among these theories, the new language allows one to explain features in a particular matroid theory using the property of its base hyperfield, thereby finding the correct generality for those features to hold.
We do the same in Theorem~\ref{thm:main_HF}, where we extend Theorem~\ref{thm:main+theorem+intro} to all matroid theories whose base hyperfield has the {\em inflation property}, first introduced in the literature with different motivations \cite{Anderson:2019, Massouros:1991}.
Unlike the special case of oriented matroids, Theorem~\ref{thm:main_HF} is new even for coherent matching fields.
In view of the connection between hyperfields and tropicalization (see the survey by Viro \cite{Viro:2010}), we ask if there can be tropicalization theories along this direction beyond the degree one case.

Taking maximal minors of a matrix over a field yields a Grassmann--Pl\"{u}cker function.
Hence one can construct matroids over common hyperfields, such as the sign hyperfield or the tropical hyperfield (corresponding to oriented matroids and valuated matroids, respectively), using the observation that those hyperfields are images of fields under hyperfield morphisms.
However, such an approach does not work in general: not every hyperfield comes from a field (indeed, the inflation property was introduced in \cite{Massouros:1991} to provide such examples), while taking determinant over hyperfields is usually multi-valued.
Hence, as a contribution to the theory, our work gives a new approach to construct matroids over hyperfields.
Moreover, we use hyperfield theory to provide a weaker statement (Corollary~\ref{coro:main_HF}) for {\em any} hyperfield that our construction gives a matroid over a possibly perturbed base hyperfield.

\subsection{Organization of the Paper}

In Section~\ref{sec:background}, we collect essential definitions and background for the central objects in this paper.
We prove Theorem~\ref{thm:main+theorem+intro} in Section~\ref{sec:main}, together with other axiomatic connections between matching fields and oriented matroids.
The study of the uniform oriented matroids arising via our construction is initiated in Section~\ref{sec:OM_MF}.
%; a more detailed example involving Ringel's non-realizable oriented matroid is given in Section~\ref{sec:ringel}.
We extend our work to matroids over hyperfields in Section~\ref{sec:hyperfield}, where we also give a brief introduction to the theory.
Finally, we conclude with a few open problems in Section~\ref{sec:conclusion}.

The follow-up of our paper will focus on the topological aspect of our construction.
In particular, we derive a topological representation of the oriented matroid in Theorem~\ref{thm:main+theorem+intro} using the polyhedral structure of the triangulation and Viro's patchworking.
Such a construction generalizes the example in Section~\ref{sec:ringel}.

\section{Background} \label{sec:background}

Throughout the paper, we fix a ground set $\ground$ of size $n$ and a set $\rows$ of size $d \leq n$.
We often identify $\ground,\rows$ with $[n] = \{1,2,\dots,n\}$ and $[d]$, hence fixing an ordering for them; we use $\{+,-,0\}$ and $\{1,-1,0\}$ for signs interchangeably, and we adopt the ordering $+,->0$ of signs.

\subsection{Oriented Matroids} \label{sec:OM}

We refer the reader to \cite{BLSWZ:1993} for a comprehensive survey on oriented matroids.
As in the case of matroids, there are multiple equivalent axiom systems for oriented matroids, but we mainly use the following definition.

\begin{definition} \label{def:chirotope}
A {\em chirotope} on $\ground$ of rank $d$ is a non-zero, alternating map $\rchi:\ground^d\rightarrow\{+,-,0\}$ that satisfies the {\em Grassmann--Pl\"{u}cker (GP) relation}:\\
For any $x_1,\ldots,x_{d-1},y_1,\ldots, y_{d+1}\in \ground$, the $d+1$ expressions 
\begin{equation*}
(-1)^k \rchi(x_1,\ldots,x_{d-1},y_k)\rchi(y_1,\ldots,\widehat{y_k},\ldots,y_{d+1}), \qquad k=1,\ldots, d+1,
\end{equation*}
either contain both a positive and a negative term, or are all zeros.
An {\em oriented matroid} on $\ground$ of rank $d$ is specified by a chirotope up to global sign change.
\end{definition}

With the ordering on $\ground$, we can specify a chirotope by its values over $\binom{\ground}{d}$ using the alternating property, here $\rchi(\sigma):=\rchi(\sigma_1,\ldots,\sigma_d)$ where $\sigma=\{\sigma_1<\ldots<\sigma_d\}$.

\begin{example} \label{ex:realizable+oriented+matroid}
 A chirotope is the generalization of the signs of maximal minors of a real matrix (oriented matroids coming from this way are said to be {\em realizable}).
  More precisely, let $A \in \RR^{d \times n}$ be a rectangular matrix with rank~$d$.
  Then
  \[
  \rchi(j_1,j_2,\dots,j_d) := \sign\det\left(a^{(j_1)},a^{(j_2)},\dots,a^{(j_d)}\right) ,
  \]
  where $a^{(j_k)}$ denotes columns of $A$, is the chirotope of an oriented matroid of rank $d$.
  Note that we can normalize $A$ to the form $(I_{d,d}|B)$ for some $B \in \RR^{d \times (n-d)}$ by multiplying with the inverse of a full-rank submatrix of $A$. 
 This incurs only a global sign change of the chirotope. 
\end{example}

%\todo[inline]{We should define 'uniform' somewhere here. }

To state an important characterization of chirotopes, we briefly recall the definition of a matroid.

\begin{definition} \label{def:matroid}
  Let $\Bcal(M)$ be a non-empty subset of $\binom{\ground}{d}$.
  Setting ${\bf e}_B:=\sum_{i\in B}{\bf e}_i\in\mathbb{R}^\ground$ for each $B\in\Bcal(M)$, $M$ is a \emph{matroid} with \emph{bases} $\Bcal(M)$ if the convex hull of $\{{\bf e}_B:B\in\Bcal(M)\}$ has only edge directions ${\bf e}_i -{\bf e}_j$ for unit vectors ${\bf e}_i, {\bf e}_j$.  
  The polytope itself is the \emph{matroid polytope} of $M$ and the \emph{rank} of $M$ is $d$.

  $\binom{\ground}{d}$ itself is the collection of bases of a matroid, known as the {\em uniform} matroid $U_{d,n}$; in such case the matroid polytope is the {\em hypersimplex}.
\end{definition}

Now if we know in advance that there is a matroid underneath, we can check whether a sign map is a chirotope locally \cite[Theorem~3.6.2]{BLSWZ:1993}.

\begin{proposition}\label{prop:3GP}
Suppose $\rchi:\ground^d\rightarrow\{+,-,0\}$ is an alternating map.
Then $\rchi$  satisfies the GP relation if and only if $\underline{\rchi}:=\{\sigma\in\binom{\ground}{d}: \rchi(\sigma)\neq 0\}$ is the collection of bases of some matroid and that $\rchi$ satisfies the {\em $3$-term GP relation}:\\
For any $x_1,x_2,y_1,y_2\in \ground,X:=\{x_3,\ldots,x_d\}\subset \ground$, the three expressions
\begin{equation}\label{eq:3GP}
\rchi(x_1,x_2,X)\rchi(y_1,y_2,X), \rchi(x_1,y_1,X)\rchi(y_2,x_2,X), \rchi(x_1,y_2,X)\rchi(x_2,y_1,X),
\end{equation}
either contain both a positive and a negative term, or are all zeros.
\end{proposition}

\subsection{Matching Fields}

A $(d,n)$-\emph{matching field} is a collection of perfect matchings $M_\sigma$'s on bipartite node sets $\rows \sqcup \sigma$, one for each $d$-subset $\sigma$ of $\ground$.
We use the terminology \emph{pd-graphs} to denote subgraphs of $K_{d,n}$ which are composed of edges of the matchings. \footnote{The `pd' stands for primal-dual as the two node classes of these bipartite graphs give rise to a dual pair of variables in the context of tropical linear programming. }.
For example, they can arise as unions (for linkage pd-graphs) or transversals (for Chow pd-graphs) of matchings. 

%\todo[inline]{MC: This definition is still not clear to me. Can we define a pd-graph as a subgraph of $K_{d,n}$ whose edges come from the matching field? If we do not want to give a definition, then we should delete the above sentence and put in the footnote something along the lines of: ``The \emph{pd-} prefix is meant to emphasize that the graph is a subgraph of $K_{d,n}$ whose edges come from the matching field."} 

\begin{example}[{Diagonal matching field~\cite{SturmfelsZelevinsky:1993}}]
  The \emph{diagonal matching field} has exactly the edges $\{(1,j_1),\dots,(d,j_d)\}$ in the matching on the ordered subset $j_1 < j_2 < \dots < j_d$ of $\ground$.
  For $d=2$ and $n=4$ this is depicted in Figure~\ref {fig:four+two+matching+field}.
\end{example}

\begin{figure}[htb]
  \begin{center}
	\resizebox{\textwidth}{!}{
  \begin{tikzpicture}[scale=0.6]

\bigraphtwofourcoord{0}{0}{1.2}{0.6}{2.5};
\draw[EdgeStyle] (v1) to (w1);
\draw[EdgeStyle] (v2) to (w2);
\bigraphtwofournodes

\bigraphtwofourcoord{4}{0}{1.2}{0.6}{2.5};
\draw[EdgeStyle] (v1) to (w1);
\draw[EdgeStyle] (v3) to (w2);
\bigraphtwofournodes

\bigraphtwofourcoord{8}{0}{1.2}{0.6}{2.5};
\draw[EdgeStyle] (v1) to (w1);
\draw[EdgeStyle] (v4) to (w2);
\bigraphtwofournodes

\bigraphtwofourcoord{12}{0}{1.2}{0.6}{2.5};
\draw[EdgeStyle] (v2) to (w1);
\draw[EdgeStyle] (v3) to (w2);
\bigraphtwofournodes

\bigraphtwofourcoord{16}{0}{1.2}{0.6}{2.5};
\draw[EdgeStyle] (v2) to (w1);
\draw[EdgeStyle] (v4) to (w2);
\bigraphtwofournodes

\bigraphtwofourcoord{20}{0}{1.2}{0.6}{2.5};
\draw[EdgeStyle] (v3) to (w1);
\draw[EdgeStyle] (v4) to (w2);
\bigraphtwofournodes
\end{tikzpicture}
	}
  \caption{The diagonal $(2,4)$-matching field.}
  \label{fig:four+two+matching+field}
	\end{center}
\end{figure}
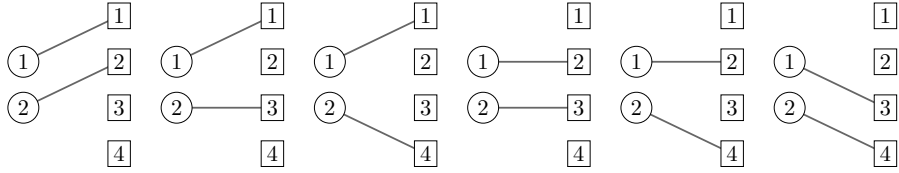

\begin{example}[Coherent matching field] \label{ex:coherent+matching+field}
  A $(d,n)$-matching field is \emph{coherent} if there is a \emph{generic} matrix $M \in \RR^{d \times n}$ such that the matching on $\rows \sqcup \sigma$ is the weight-maximal perfect matching induced by the entries of $M$ as weights on $K_{d,n}$.
  The genericity here means that the weight-maximal matching is unique.
  A diagonal matching field is coherent as it is induced by the weight matrix $((i-1) \cdot (j-1))_{(i,j) \in [d] \times [n]}$.
\end{example}

\begin{example}[Linkage matching field] \label{ex:linkage+matching+field}
  A matching field is \emph{linkage} if it fulfills a local compatibility condition.
  Namely, for every $(d+1)$-subset $\tau$ of $\ground$, the union of the matchings on $\tau$ is a spanning tree on $\rows \sqcup \tau$; such a tree is a \emph{linkage pd-graph}.
  These graphs were called `linkage covectors' in~\cite{LohoSmith:2020} but we show later that they can give rise to signed circuits. 
  Note that each node in $\rows$ of a linkage pd-graph has to have degree $2$ by a counting argument.
  Figure~\ref{fig:non+right+linkage+field} shows a linkage matching field which is the pointed extension (see Section~\ref{sec:matroid+subdivions+from+triang}) of the non-linkage matching field depicted in Figure~\ref{fig:non+linkage+field}.
  It was used in~\cite[Proposition~2.3]{SturmfelsZelevinsky:1993} to show the existence of non-coherent linkage matching fields.
\end{example}

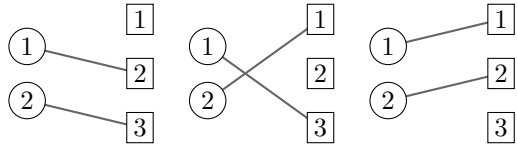
\begin{figure}[htb]
  \begin{center}
  \begin{tikzpicture}[scale=0.6]

\bigraphthreetwocoord{0}{0}{1.2}{0.6}{2.5};
\draw[EdgeStyle] (v2) to (w1);
\draw[EdgeStyle] (v3) to (w2);
\bigraphthreetwonodes

\bigraphthreetwocoord{4}{0}{1.2}{0.6}{2.5};
\draw[EdgeStyle] (v1) to (w2);
\draw[EdgeStyle] (v3) to (w1);
\bigraphthreetwonodes

\bigraphthreetwocoord{8}{0}{1.2}{0.6}{2.5};
\draw[EdgeStyle] (v1) to (w1);
\draw[EdgeStyle] (v2) to (w2);
\bigraphthreetwonodes

\end{tikzpicture}
  \caption{This is the smallest matching field which is not linkage. }
  \label{fig:non+linkage+field}
	\end{center}
\end{figure}

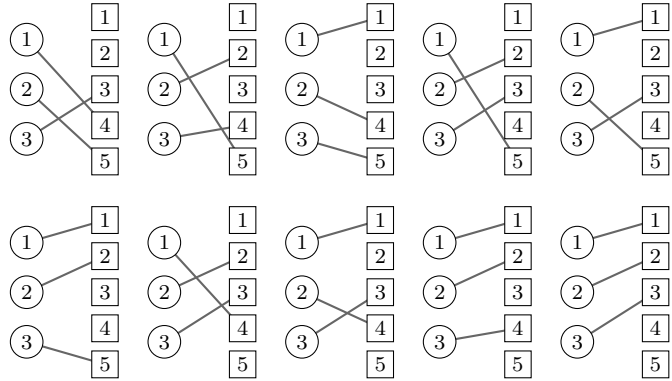
\begin{figure}[htb]
  \begin{center}
  \begin{tikzpicture}[scale=0.3]

  \footnotesize
  
  \bigraphthreefivecoord{0}{0}{1.6}{2.2}{3.47};
  \draw[EdgeStyle] (w1) to (v4);
  \draw[EdgeStyle] (w2) to (v5);
  \draw[EdgeStyle] (w3) to (v3);
  \bigraphthreefivenodes

  \bigraphthreefivecoord{6.1}{0}{1.6}{2.2}{3.47};
  \draw[EdgeStyle] (w1) to (v5);
  \draw[EdgeStyle] (w2) to (v2);
  \draw[EdgeStyle] (w3) to (v4);
  \bigraphthreefivenodes

  \bigraphthreefivecoord{12.2}{0}{1.6}{2.2}{3.47};
  \draw[EdgeStyle] (w1) to (v1);
  \draw[EdgeStyle] (w2) to (v4);
  \draw[EdgeStyle] (w3) to (v5);
  \bigraphthreefivenodes

  \bigraphthreefivecoord{18.3}{0}{1.6}{2.2}{3.47};
  \draw[EdgeStyle] (w1) to (v5);
  \draw[EdgeStyle] (w2) to (v2);
  \draw[EdgeStyle] (w3) to (v3);
  \bigraphthreefivenodes

  \bigraphthreefivecoord{24.4}{0}{1.6}{2.2}{3.47};
  \draw[EdgeStyle] (w1) to (v1);
  \draw[EdgeStyle] (w2) to (v5);
  \draw[EdgeStyle] (w3) to (v3);
  \bigraphthreefivenodes

%%%%%%
  
  \bigraphthreefivecoord{0}{-9}{1.6}{2.2}{3.47};
  \draw[EdgeStyle] (w1) to (v1);
  \draw[EdgeStyle] (w2) to (v2);
  \draw[EdgeStyle] (w3) to (v5);
  \bigraphthreefivenodes

  \bigraphthreefivecoord{6.1}{-9}{1.6}{2.2}{3.47};
  \draw[EdgeStyle] (w1) to (v4);
  \draw[EdgeStyle] (w2) to (v2);
  \draw[EdgeStyle] (w3) to (v3);
  \bigraphthreefivenodes

  \bigraphthreefivecoord{12.2}{-9}{1.6}{2.2}{3.47};
  \draw[EdgeStyle] (w1) to (v1);
  \draw[EdgeStyle] (w2) to (v4);
  \draw[EdgeStyle] (w3) to (v3);
  \bigraphthreefivenodes

  \bigraphthreefivecoord{18.3}{-9}{1.6}{2.2}{3.47};
  \draw[EdgeStyle] (w1) to (v1);
  \draw[EdgeStyle] (w2) to (v2);
  \draw[EdgeStyle] (w3) to (v4);
  \bigraphthreefivenodes

  \bigraphthreefivecoord{24.4}{-9}{1.6}{2.2}{3.47};
  \draw[EdgeStyle] (w1) to (v1);
  \draw[EdgeStyle] (w2) to (v2);
  \draw[EdgeStyle] (w3) to (v3);
  \bigraphthreefivenodes

\end{tikzpicture}
  \caption{This is the smallest linkage matching field which is not polyhedral. }
  \label{fig:non+right+linkage+field}
	\end{center}
\end{figure}

\subsection{Triangulations of $\ssimplex_{d-1}\times\ssimplex_{n-1}$ and Polyhedral Matching Fields} \label{sec:triangulation}

We give a brief introduction to the necessary polyhedral notions and refer the reader to~\cite{DeLoeraRambauSantos:2010} for more details. 
We denote the $(k-1)$-simplex by $\ssimplex_{k-1}$, that is the convex hull of $k$ affinely independent points. 
Even if we are mainly interested in combinatorial properties, we use the standard embedding $\ssimplex_{k-1} = \conv\{{\bf e}_i \colon i \in [k]\} \subset \RR^k$.
The product $\ssimplex_{d-1}\times\ssimplex_{n-1}$ of a $(d-1)$-simplex and an $(n-1)$-simplex is the convex hull
\[
\ssimplex_{d-1}\times\ssimplex_{n-1} = \conv\{ ({\bf e}_i,{\bf e}_j) \colon i \in [d],j \in [n] \}\subset\mathbb{R}^d\times\mathbb{R}^n \enspace .
\]

A \emph{triangulation} of $\ssimplex_{d-1}\times\ssimplex_{n-1}$ is a collection of full-dimensional simplices $\mathcal{T}$ whose vertices form a subset of the vertices of $\ssimplex_{d-1}\times\ssimplex_{n-1}$, such that
\begin{enumerate}
\item (Union) the union of all simplices in $\mathcal{T}$ is $\ssimplex_{d-1}\times\ssimplex_{n-1}$, 
\item (Intersection) two simplices in $\mathcal{T}$ intersect in a common face. 
\end{enumerate}

Each simplex in $\mathcal{T}$ gives rise to a subgraph of the complete bipartite graph on the vertex set $\rows \sqcup \ground$ by identifying a vertex $({\bf e}_i,{\bf e}_j)$ with an edge.
%We call these graphs the $\mathcal{T}$-graphs. 
Rephrasing the definition of a triangulation in terms of graphs leads to the following characterization.

\begin{proposition}[{\cite[Proposition 7.2]{ArdilaBilley:2007}}] \label{prop:char-triang} \leavevmode
  A set of graphs on the bipartite node set $\rows \sqcup \ground$ encodes the maximal simplices of a triangulation of $\ssimplex_{d-1}\times\ssimplex_{n-1}$ if and only if:
  \begin{enumerate} 
  \item Each graph is a spanning tree.
  \item For each tree $G$ and each edge $e$ of $G$, either $G - e$ has an isolated node or there is another tree $H$ containing $G - e$.
  \item \label{item:compatibility} If two trees $G$ and $H$ contain perfect matchings on $I \sqcup J$ for $I \subseteq \rows$ and $J \subseteq \ground$ with $|J| = |I|$, then the matchings agree. 
  \end{enumerate}
\end{proposition}

  Starting from a triangulation of $\ssimplex_{d-1}\times\ssimplex_{n-1}$, we collect all $\rows$-saturating matchings (those covering all nodes in $\rows$) that appear as subgraphs of the trees corresponding to the simplices.
  By (3) of Proposition~\ref{prop:char-triang}, there is at most one matching for each $d$-subset of $\ground$. 
  Considering the barycentre of the subpolytope corresponding to $K_{\rows,\sigma}$ shows that there is indeed a matching for each $d$-subset \cite[Proposition~2.5]{LohoSmith:2020}.

\begin{definition}[Polyhedral matching field] \label{ex:polyhedral+matching+field}
  A \emph{polyhedral matching field} is the collection of all $\rows$-saturating matchings appearing in the trees encoding a triangulation of $\ssimplex_{d-1}\times\ssimplex_{n-1}$.
\end{definition}

 Every coherent matching field is the polyhedral matching field extracted from a regular triangulation of $\ssimplex_{d-1}\times\ssimplex_{n-1}$, induced by the same weight matrix.
  On the other hand, all polyhedral matching fields are linkage \cite[Section~3.1]{LohoSmith:2020}.
  We see in Example~\ref{ex:SZ_example} that the linkage matching field in Figure~\ref{fig:non+right+linkage+field} is not polyhedral.

\begin{example}
  Figure~\ref{fig:four+two+trees} depicts four trees comprising a triangulation of $\ssimplex_{1}\times\ssimplex_{3}$.
  The $2 \times 2$ matchings arising in these trees yield the matching field shown in Figure~\ref{fig:four+two+matching+field}.
\end{example}

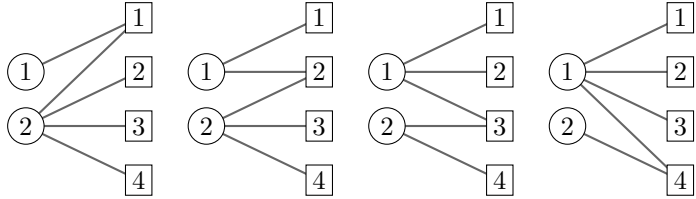
\begin{figure}[htb]
  \begin{center}
  \begin{tikzpicture}[scale=0.6]

\bigraphtwofourcoord{0}{0}{1.2}{0.6}{2.5};
\draw[EdgeStyle] (v1) to (w1);
\draw[EdgeStyle] (v1) to (w2);
\draw[EdgeStyle] (v2) to (w2);
\draw[EdgeStyle] (v3) to (w2);
\draw[EdgeStyle] (v4) to (w2);
\bigraphtwofournodes

\bigraphtwofourcoord{4}{0}{1.2}{0.6}{2.5};
\draw[EdgeStyle] (v1) to (w1);
\draw[EdgeStyle] (v2) to (w1);
\draw[EdgeStyle] (v2) to (w2);
\draw[EdgeStyle] (v3) to (w2);
\draw[EdgeStyle] (v4) to (w2);
\bigraphtwofournodes

\bigraphtwofourcoord{8}{0}{1.2}{0.6}{2.5};
\draw[EdgeStyle] (v1) to (w1);
\draw[EdgeStyle] (v2) to (w1);
\draw[EdgeStyle] (v3) to (w1);
\draw[EdgeStyle] (v3) to (w2);
\draw[EdgeStyle] (v4) to (w2);
\bigraphtwofournodes

\bigraphtwofourcoord{12}{0}{1.2}{0.6}{2.5};
\draw[EdgeStyle] (v1) to (w1);
\draw[EdgeStyle] (v2) to (w1);
\draw[EdgeStyle] (v3) to (w1);
\draw[EdgeStyle] (v4) to (w1);
\draw[EdgeStyle] (v4) to (w2);
\bigraphtwofournodes

\end{tikzpicture}
  \caption{Trees corresponding to maximal simplices in a triangulation of $\ssimplex_{1}\times\ssimplex_{3}$. }
  \label{fig:four+two+trees}
	\end{center}
\end{figure}

Using~\cite[Theorem~3.16]{LohoSmith:2020}, one can construct the collection $\Ical:=\Ical(\ssimplex_{d-1}\times\ssimplex_{n-1})$ of trees such that each node in $\rows$ has degree at least $2$ by iteratively taking linkage pd-graphs. 
On the other hand, for $n > d$, all the matchings of the matching field occur in at least one of these trees as one can see from the barycentre construction. 
These trees have another nice property which will be useful later. 

\begin{lemma} \label{lem:tranversal+full+rank}
Every tree $T$ in $\Ical$ contains an $\rows$-saturating matching.
\end{lemma}
\begin{proof}
  Since each node in $\rows$ has degree bigger than $1$, there is a leaf $\ell$ of $T$ in~$\ground$.
  Include the incident edge $(i,\ell)$ in the matching and iterate the procedure with $T-\{i,\ell\}$.
  This is possible as all remaining nodes in $\rows$ are still of degree at least~$2$.
\end{proof}

\smallskip

A special class of polyhedral matching fields are the \emph{pointed} ones (following the terminology in~\cite{SturmfelsZelevinsky:1993}), which comprise the full information of a triangulation of $\ssimplex_{d-1}\times\ssimplex_{n-1}$.
We augment the ground set $\ground$ by a copy $\widetilde{\rows}$ of $\rows$ to obtain a ground set $\augground$ of size $n+d$, and we set all elements of $\widetilde{\rows}$ to be smaller than all elements of $\ground$.
To take the full information of all trees into account, for each tree in $\mathcal{T}$ we add edges between each node in $\rows$ to its copy $\widetilde{\rows}$, yielding a set $\widetilde{\mathcal{T}}$ of trees on $\rows \sqcup \augground$.

\begin{definition} \label{def:pointed+polyhedral}
  The \emph{pointed polyhedral matching field} associated with a triangulation $\mathcal{T}$ of $\ssimplex_{d-1}\times\ssimplex_{n-1}$ is the collection of $\rows$-saturating matchings on $\rows \sqcup \augground$ in the trees in $\widetilde{\mathcal{T}}$. 
\end{definition}

\begin{figure}[htb]
  \begin{center}
  \begin{tikzpicture}[scale=0.3]

  \footnotesize
  
  \bigraphtwofivecoord{0}{0}{1.6}{1.2}{3.47};
  \draw[EdgeStyle] (w1) to (v1);
  \draw[EdgeStyle] (w2) to (v2);
  \bigraphtwofivenodes

  \bigraphtwofivecoord{6.1}{0}{1.6}{1.2}{3.47};
  \draw[EdgeStyle] (w1) to (v1);
  \draw[EdgeStyle] (w2) to (v3);
  \bigraphtwofivenodes

  \bigraphtwofivecoord{12.2}{0}{1.6}{1.2}{3.47};
  \draw[EdgeStyle] (w1) to (v4);
  \draw[EdgeStyle] (w2) to (v1);
  \bigraphtwofivenodes

  \bigraphtwofivecoord{18.3}{0}{1.6}{1.2}{3.47};
  \draw[EdgeStyle] (w1) to (v1);
  \draw[EdgeStyle] (w2) to (v5);
  \bigraphtwofivenodes

  \bigraphtwofivecoord{24.4}{0}{1.6}{1.2}{3.47};
  \draw[EdgeStyle] (w1) to (v3);
  \draw[EdgeStyle] (w2) to (v2);
  \bigraphtwofivenodes

%%%%%%

  \bigraphtwofivecoord{0}{-9}{1.6}{1.2}{3.47};
  \draw[EdgeStyle] (w1) to (v4);
  \draw[EdgeStyle] (w2) to (v2);
  \bigraphtwofivenodes

  \bigraphtwofivecoord{6.1}{-9}{1.6}{1.2}{3.47};
  \draw[EdgeStyle] (w1) to (v2);
  \draw[EdgeStyle] (w2) to (v5);
  \bigraphtwofivenodes

  \bigraphtwofivecoord{12.2}{-9}{1.6}{1.2}{3.47};
  \draw[EdgeStyle] (w1) to (v4);
  \draw[EdgeStyle] (w2) to (v3);
  \bigraphtwofivenodes

  \bigraphtwofivecoord{18.3}{-9}{1.6}{1.2}{3.47};
  \draw[EdgeStyle] (w1) to (v3);
  \draw[EdgeStyle] (w2) to (v5);
  \bigraphtwofivenodes

  \bigraphtwofivecoord{24.4}{-9}{1.6}{1.2}{3.47};
  \draw[EdgeStyle] (w1) to (v4);
  \draw[EdgeStyle] (w2) to (v5);
  \bigraphtwofivenodes
  
\end{tikzpicture}
  \caption{The pointed polyhedral matching field encoding the triangulation of the prism in Figure~\ref{fig:regular-subdivision-prism} as visualization of Definition~\ref{def:pointed+polyhedral}. }
  \label{fig:pointed+polyhedral+for+triangulation}
	\end{center}
\end{figure}
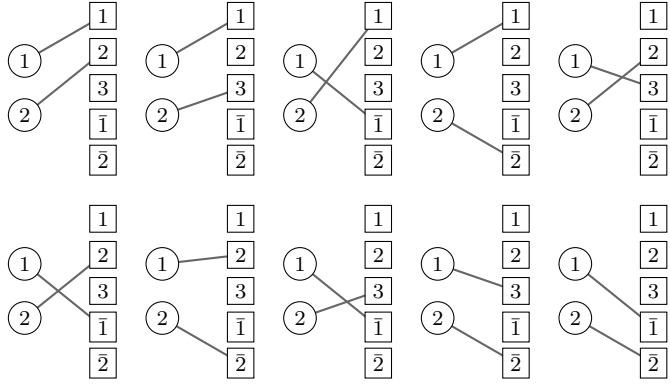

The discussion after \cite[Theorem 3.16]{LohoSmith:2020} shows that the latter construction actually yields a correspondence between triangulations and matching fields.

One can see that pointed polyhedral matching fields are actually polyhedral matching fields by extending the original triangulation of $\ssimplex_{d-1}\times\ssimplex_{n-1}$ to a triangulation of $\ssimplex_{d-1}\times\ssimplex_{n+d-1}$ by using a placing triangulation as in~\cite[Lemma~4.3.2]{DeLoeraRambauSantos:2010}.
Note that, on the other hand, each polyhedral matching field is a sub-matching field of a pointed polyhedral matching field.
We use both points of view as they allow different constructions as we see in Section~\ref{sec:topes+covectors} and the follow-up of our paper.
It is similar to the relation between transversal matroids and fundamental transversal matroids, and it is reminiscent of the correspondence between matroids and linking systems~\cite{Schrijver:1979}.

\bigskip

%\todo[inline]{YCH: Do we need all the discussion about Cayley trick and tropical OM here? GL: I propose to leave it. }

Cutting appropriately through a triangulation of $\ssimplex_{d-1}\times\ssimplex_{n-1}$ yields a \emph{fine mixed subdivision} of $n\ssimplex_{d-1}$.
This is formalized as the \emph{Cayley trick}, see~\cite{Santos:2005}. 
A direct way to identify the polyhedral pieces in such a subdivision of $n\ssimplex_{d-1}$ is the following.
For each tree $G$ corresponding to a simplex in the triangulation $\mathcal{T}$, we form the Minkowski sum
\[
\sum_{j \in \ground} \conv \{ {\bf e}_i \colon i \in \mathcal{N}_G(j) \} \enspace ,
\]
where $\neighbourhood_G(j)$ is the neighbourhood of an element $j \in \ground$ in $G$.
The collection of these Minkowski sums tiles the dilated simplex $n\ssimplex_{d-1}$.

The dual polyhedral complexes to these mixed subdivisions were introduced as \emph{tropical pseudohyperplane arrangements} in~\cite{ArdilaDevelin:2009}.
This draws from the correspondence between tropical hyperplane arrangements and regular subdivisions of products of simplices established in~\cite{DevelinSturmfels:2004}.
After starting from an axiomatic study of \emph{tropical oriented matroids} in~\cite{ArdilaDevelin:2009}, it was subsequently shown that these combinatorial objects are indeed \emph{cryptomorphic} to subdivisions of $\ssimplex_{d-1}\times\ssimplex_{n-1}$ and tropical pseudohyperplane arrangements, partially in~\cite{OhYoo:2011} and finished in~\cite{Horn1}.

\begin{figure}[htb]
  \centering
  \includegraphics[scale=0.5]{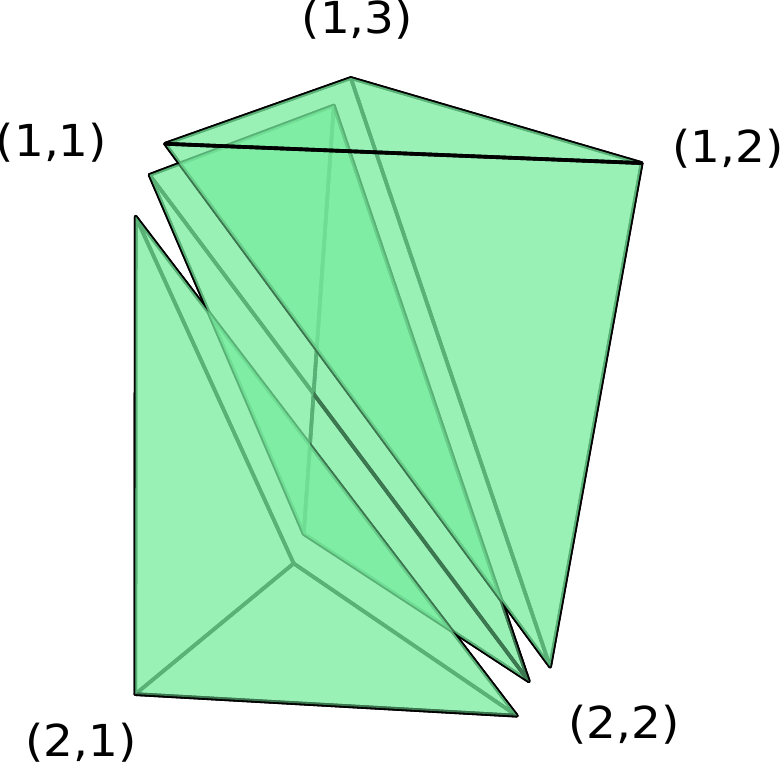}
  \caption{A triangulation of $\ssimplex_1\times\ssimplex_2$. The vertices are labeled by the corresponding edges in $K_{2,3}$. This picture was created with \texttt{polymake} \cite{DMV:polymake}.}
  \label{fig:regular-subdivision-prism}
\end{figure}

\section{Polyhedral Matching Fields Induce\\ Oriented Matroids}\label{sec:main}

\subsection{Matroid Subdivisions from Triangulations} \label{sec:matroid+subdivions+from+triang}

We first briefly recall the definition of a matroid subdivision.

\begin{definition} 
A collection of matroids is a {\em matroid subdivision} of a matroid $M$ if they have the same ground set and rank as $M$, and their matroid polytopes (together with their faces) form a subdivision of the matroid polytope of $M$.
\end{definition}

In the proof of~\cite[Proposition~2.2]{Speyer:2008}, the argument for (1) $\Rightarrow$ (2) does not depend on the actual tropical Pl{\"u}cker vector but only on the induced subdivision of the octahedron.
Together with the equivalence with (3), this shows the following result, which also applies to non-regular subdivisions. 
The proof of our main theorem in the next section is in a similar spirit. 

\begin{proposition}
A polyhedral subdivision of matroid polytope is a matroid subdivision if and only if the induced subdivision of each octahedron face is a matroid subdivision.
\end{proposition}

The construction of matroid subdivisions from triangulations of products of simplices goes back to~\cite{Kapranov:1993} and was further refined in~\cite{FinkRincon:2015,HerrmannJoswigSpeyer:2014,Rincon:2013}.
We take a more combinatorial point of view described in~\cite{LohoSmith:2020}.
Given a subgraph $G$ of $K_{\rows,\ground}$, the {\em transversal matroid} $\Mcal(G)$
is the matroid whose bases are the subsets of $\ground$ which are matched by some $\rows$-saturating matching of $G$.
Starting from the set of all trees corresponding to the full-dimensional simplices in a triangulation of $\ssimplex_{d-1}\times\ssimplex_{n-1}$, we restrict to the subset $\Ical$ of all trees where each node in $\rows$ has degree at least $2$.
By Lemma~\ref{lem:tranversal+full+rank}, the transversal matroid of each such tree has rank $d$.

\begin{theorem}[{\cite[\S 5.2]{FinkRincon:2015},\cite[\S 2]{HerrmannJoswigSpeyer:2014}, \cite{Kapranov:1993,Rincon:2013}}] \label{thm:matroid+subdivision+from+triangulation}
  The matroids $\Mcal(T)$ ranging over all $T \in \Ical$ form
  a matroid subdivision of $U_{d,n}$. 
\end{theorem}

Note that the subdivision of $U_{d,n}$ is regular if and only if the corresponding subdivision of $\ssimplex_{d-1}\times\ssimplex_{n-1}$ is regular, see~\cite[\S 2]{HerrmannJoswigSpeyer:2014} for more details. 

\begin{example}[Matroid subdivisions of an octahedron] \label{ex:matroid+subdivision+octahedron}
  An octahedron is the matroid polytope of the uniform matroid $U_{2,4}$.
  There are exactly two types of matroid subdivisions of an octahedron.
  The trivial subdivision has only the octahedron as a cell.
  Apart from that, one can subdivide it into two pyramids as depicted in Figure~\ref{fig:matroid-subdivision}; there are three such subdivisions.
  While one can also subdivide the octahedron into four tetrahedra, it is not a matroid subdivision by the edge direction criterion.
  \smallskip

  The non-trivial subdivision in Figure~\ref{fig:matroid-subdivision} arises through the construction of Theorem~\ref{thm:matroid+subdivision+from+triangulation} from Example~\ref{ex:polyhedral+matching+field} and Figure~\ref{fig:four+two+trees}.
  We restrict our attention to the two trees with $\rows$-degree vector $[2,3]$ and $[3,2]$.
  The combinatorial Stiefel map gives the two sets $\{12,13,14,23,24\}$ and $\{13,14,23,24,34\}$, which form a matroid subdivision of $U_{2,4}$.
\end{example}

\begin{figure}[htb]
  \centering
  \includegraphics[scale=0.5]{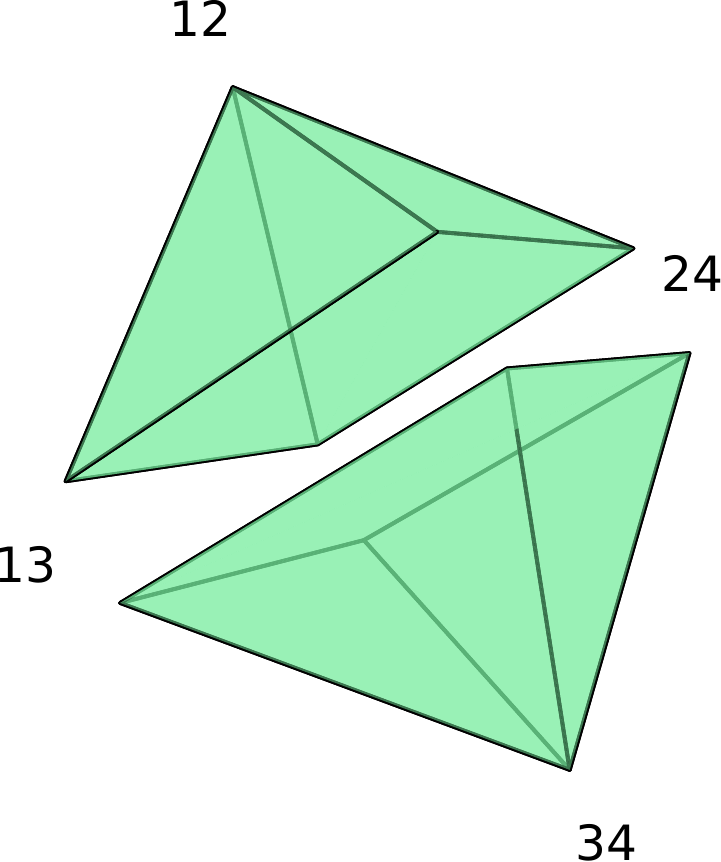}
  \caption{The matroid subdivision described in Example~\ref{ex:matroid+subdivision+octahedron}.}
  \label{fig:matroid-subdivision}
\end{figure}

Recall from Section~\ref{sec:triangulation} the construction of the set $\widetilde{\mathcal{T}}$ of trees on $\rows \sqcup \augground$. 
For each tree $T$ in $\widetilde{\mathcal{T}}$, we get an induced transversal matroid $\modMcal(T)$ on $\augground = \widetilde{\rows} \cup \ground $. 
In particular, the trees give rise to a matroid subdivision of $U_{d,d+n}$. 

\begin{corollary} \label{cor:matroid+subdivision+from+pointed}
  The matroids $\modMcal(T)$ ranging over all trees $T$ comprising a triangulation of $\ssimplex_{d-1}\times\ssimplex_{n-1}$ form the maximal cells for a matroid subdivision of~$U_{d,d+n}$.
\end{corollary}

There is another way to think about this subdivision more geometrically.
Following \cite[\S4]{Rincon:2013}, all cells in this subdivision are principal (or fundamental) transversal matroids with fundamental basis $\widetilde{\rows}$.
Indeed, one can obtain this subdivision concretely as follows:
In the matroid polytope of $U_{d,d+n}$, the vertices adjacent to $\ee_{\widetilde{\rows}}$ are in natural bijection with the vertices of $\ssimplex_{d-1}\times\ssimplex_{n-1}$ and their convex hull is $\left(-\ssimplex_{d-1}\right)\times\ssimplex_{n-1}$ up to translation.
Triangulate the convex hull as the initial triangulation, and cone over the cells from $\ee_{\widetilde{\rows}}$.
The matroid subdivision in Corollary~\ref{cor:matroid+subdivision+from+pointed} is the intersection of these cones with the matroid polytope.

\subsection{Proof of Theorem~\ref{thm:main+theorem+intro}}

Fix a triangulation of $\ssimplex_{d-1}\times\ssimplex_{n-1}$, and as before let $\mathcal{I}$ denote the trees of the triangulation of whose nodes in $\rows$ each have degree at least 2. Let $(M_{\sigma})$ denote the polyhedral matching field extracted from $\mathcal{I}$. Fix a sign matrix $A\in\{+,-,0\}^{\rows\times\ground}$ such that an entry is non-zero whenever the corresponding edge appeared in some matching from $(M_{\sigma})$ (the collection of edges appeared in the matchings is the {\em support} of the matching field).

%For the next definition we assume the elements of both $\rows$ and $\ground$ are totally ordered.
%\todo{We have already stated this as the standing assumption in Section 2...}

\begin{definition}[Sign of a matching] \label{def:sign+of+a+matching}
Let $M_{\sigma}$ be a perfect matching between $\rows$ and $\sigma\subset\ground$ in $K_{\rows,\ground}$.
Define $\sign(M_{\sigma})$ to be the sign of the permutation
\[
[d]\longrightarrow \rows\longrightarrow \sigma \longrightarrow[d],
\]
where the first map is the given and last maps are order-preserving bijections, with $\sigma$ inheriting the order from $\ground$, and
the middle map is the bijection determined by $M_{\sigma}$. 

\end{definition}

%\todo[inline]{Introduce the notion of a support of a matching field and state the slightly more general version of Theorem A for which the support of the sign matrix contains the support of the matching field. }

With the terminology in place we restate our first main Theorem~\ref{thm:main+theorem+intro}. 

\begin{theorem} \label{thm:main}
The sign map $\rchi:\binom{\ground}{d}\rightarrow\{+,-\}$ given by
\begin{equation} \label{eq:sign+map+matching+field}
  \sigma\mapsto\sign(M_\sigma)\prod_{e\in M_\sigma}A_e \enspace ,
\end{equation}
is the chirotope of a uniform oriented matroid. 
\end{theorem}

Let $T$ be a tree in $\Ical$ and let $A(T)$ be any matrix obtained by setting every entry of $A$ not in $T$ to $0$ and replacing every remaining non-zero entry by an arbitrary real number of the same sign.
Furthermore, let $\rchi_{T}$ be the map $\rchi$ restricted to the bases of the transversal matroid $\Mcal(T)$.

\begin{lemma} \label{lem:TM_real}
  The map $\rchi_{T}$ is the chirotope of the rank $d$ oriented matroid realized by $A(T)$.
\end{lemma}

\begin{proof}
  By Lemma~\ref{lem:tranversal+full+rank}, $\rchi_T$ is non-zero.
If $\sigma\in\binom{\ground}{d}$ is not a basis of the transversal matroid, then there are no perfect matchings between $\rows$ and $\sigma$ in $T$, thus $\det(A(T)_{|\sigma})$ is $0$. 
Otherwise there is a unique matching, namely $M_\sigma$, between $\rows$ and $\sigma$ in $T$. 
This matching corresponds to the unique non-zero term in the expansion of $\det(A(T)_{|\sigma})$, whose sign coincides with $\rchi(\sigma)$ by definition.
\end{proof}

\begin{theorem}\label{thm:local_global}
Let $M_1,\ldots, M_k$ be a matroid subdivision of some matroid $M$. 
Let $\rchi$ be an alternating sign map supported on the bases $\Bcal(M)$.
Suppose each restriction $\rchi_i:\Bcal(M_i)\rightarrow\{+,-\}$ is a chirotope. 
Then $\rchi$ itself is a chirotope.
\end{theorem}

\begin{proof}
Suppose $\rchi$ is not a chirotope. 
Since $\underline{\rchi}$ is by assumption a matroid, by Proposition~\ref{prop:3GP}, there exist $x_1,x_2,y_1,y_2,x_3,\ldots,x_d\in \ground$ such that (\ref{eq:3GP}) is violated. 
It is routine to check that all $d+2$ elements must be distinct here, and the six $d$-element subsets involved in (\ref{eq:3GP}) form three ``antipodal'' pairs. 
Moreover, at least one such pair consists of bases of $M$ as at least a term in (\ref{eq:3GP}) is non-zero.

\smallskip

Construct a vector ${\bf c}\in\mathbb{R}^{\ground}$ where ${\bf c}(x_1)={\bf c}(x_2)={\bf c}(y_1)={\bf c}(y_2)=1$, ${\bf c}(x_i)=0, \forall i=3,\ldots,d$, and ${\bf c}(\ell)$'s are sufficiently large numbers for every other $\ell\in \ground$. 
Consider the face $F_{\bf c}$ of the matroid polytope of $M$ that minimizes ${\bf x}\mapsto {\bf c}^{\top}{\bf x}$. 
A vertex of the matroid polytope of $M$ achieves the minimum if and only if it corresponds to one of the aforementioned subsets while being a basis of $M$. 
$F_{\bf c}$~is a matroid polytope and the global matroid subdivision restricts to a matroid subdivision of it. 
We distinguish two cases depending on the shape of $F_{\bf c}$:

\smallskip

Case I: $F_{\bf c}$ is an octahedron, i.e., all three terms of (\ref{eq:3GP}) are non-zero. 
There are four possible matroid subdivisions of an octahedron as discussed in Ex.~\ref{ex:matroid+subdivision+octahedron}, the trivial one and three that subdivide it into two pyramids. 
In either case, there is a 3-dimensional cell $C'$ that contains at least two pairs of antipodal vertices, pick a full-dimensional cell $C:=C_{M_i}$ of the global matroid subdivision that contains $C'$. 
Then restricting $\rchi$ to $M_i$  means that we are setting at most one term out of the three terms in (\ref{eq:3GP}) to zero, which still yields a violation of the 3-term GP relation in~$\rchi_{M_i}$.

\smallskip

Case II: $F_{\bf c}$ is not an octahedron (thus either a square or a pyramid). 
The only possible matroid subdivision of $F_{\bf c}$ is the trivial one. 
Similar to Case I, we pick a cell $C:=C_{M_i}$ containing $F_{\bf c}$, and $\rchi_{M_i}$ will violate (\ref{eq:3GP}) exactly like $\rchi$.
\end{proof}

\begin{proof}[{Proof of Theorem~\ref{thm:main}}]
By Corollary~\ref{cor:matroid+subdivision+from+pointed}, there is a matroid subdivision of the uniform matroid $U_{d,n}$ consisting of transversal matroids, one for each spanning tree in $\Ical$.
By Lemma~\ref{lem:TM_real}, the sign maps $\rchi_T$'s are chirotopes. 
Hence $\rchi$ itself is a chirotope by Theorem~\ref{thm:local_global}.
\end{proof}

\begin{example} \label{ex:diagonal+2+4+signs}
  The diagonal matching field depicted in Figure~\ref{fig:four+two+matching+field} with the sign-matrix
  \[
  \begin{pmatrix}
    + & - & + & -\\
    + & - & - & +
  \end{pmatrix}
  \]
  gives rise to the chirotope
  \[
  (12,-),(13,-),(14,+),(23,+),(24,-),(34,+) \enspace .
  \]
  This is shown in Figure~\ref{fig:octahedron+signs+trees}. 
\end{example}

\begin{figure}[htb]
    \begin{center}
 \input{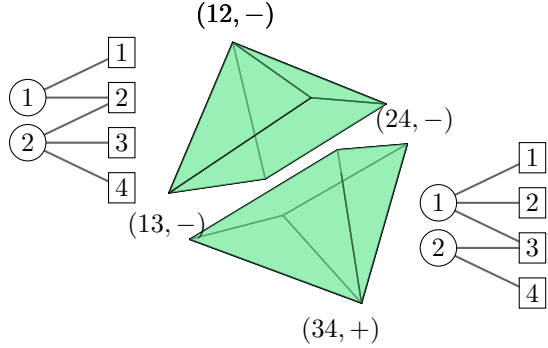}
  \caption{A matroid subdivision with signs induced by the trees and signs from Example~\ref{ex:diagonal+2+4+signs}. }
  \label{fig:octahedron+signs+trees}
    \end{center}
\end{figure}

\begin{remark}
  Instead of the full polytope $\ssimplex_{d-1} \times \ssimplex_{n-1}$, one could equally start with a subpolytope of it.
  The whole construction then yields a chirotope for the transversal matroid of the subgraph of $K_{d,n}$ corresponding to the subpolytope.
  However, it is no restriction to study only the full polytope as each triangulation can be extended to the full product of simplices.
  This is further explained in~\cite{DeLoeraRambauSantos:2010} and (also with signs) in~\cite{Loho:2016}.
\end{remark}

\begin{example}
The assumption in Theorem~\ref{thm:main} that the support of $A$ contains the support of the matching field is important.
Consider the following sign matrix together with the $(3,5)$-diagonal matching field:
$$
\begin{pmatrix}
+ & + & 0 & 0 & 0\\
0 & + & 0 & + & 0\\
0 & 0 & + & 0 & +
\end{pmatrix}.
$$
The support of the induced $\rchi$ is $\{123, 125, 145, 245\}$, which is not the collection of bases of any matroid.
\end{example}

\begin{example} \label{ex:local_global_converse}
We note that the converse of Theorem~\ref{thm:local_global} is not true in general.
Consider the assignment of signs
  \[
  (12,+),(13,+),(14,+),(23,+),(24,-),(34,-) \enspace ,
  \]
together with the matroid subdivision as in Figure~\ref{fig:octahedron+signs+trees}.
It is routine to check that the assignment is a valid chirotope.
However, when restricted to the upper piece (the pyramid containing the basis $12$ but not $34$), the assignment no longer satisfies the 3-term GP relation.
\end{example}

\begin{example} \label{ex:SZ_example}
  As the linkage property was introduced in~\cite{SturmfelsZelevinsky:1993} in analogy with the basis exchange axiom for matroids, it is tempting to assume that a linkage matching field gives rise to an oriented matroid, generalizing our construction for polyhedral matching fields. 
  However, the linkage but not polyhedral matching field depicted in Figure~\ref{fig:non+right+linkage+field} shows that this is not correct.
  Take the sign matrix to be all positive. We have the following matchings (each tuple is a subset $\{i_1<i_2<i_3\}\in\binom{[5]}{3}$ followed by the permutation $\sigma(i_1),\sigma(i_2),\sigma(i_3)$ of $[3]$ and the sign of the permutation):

\begin{center}
\begin{tabular}{lll}
  $125$&$ 123$&$ +$\\
  $135$&$ 132$&$ -$\\
  $145$&$ 123$&$ +$\\
  $235$&$ 231$&$ +$\\
  $245$&$ 231$&$ +$\\
  $345$&$ 231$&$ +$.\\
\end{tabular}
\end{center}

Take $(x_1,x_2,y_1,y_2,X)=(1,2,3,4,\{5\})$ in~\eqref{eq:3GP}.
Since all products
\begin{align*}
  \rchi(1,2,5)\cdot\rchi(3,4,5)=+\cdot + = + \\
  \rchi(1,3,5)\cdot\rchi(4,2,5)=-\cdot - = + \\
  \rchi(1,4,5)\cdot\rchi(2,3,5)=+\cdot + = +
\end{align*}
  are positive, this contradicts the $3$-term GP relation. 

  \smallskip
  
We elaborate more on this.
A linkage matching field also determines a collection of spanning trees \cite{LohoSmith:2020}, but those trees might contain new matchings.
For example, in \cite[Fig.~11]{LohoSmith:2020}, the top tree and the left tree contain different matchings on $3456$.
In such a case, while each tree still yields a chirotope, their values need not agree.
The values on some of the bases correspond to distinct matchings which means geometrically that these chirotopes can not be ``glued'' together.
\end{example}

\subsection{Compatibility of chirotopes} \label{sec:compatibility+chirotopes}

It is interesting to ask when does the converse of Theorem~\ref{thm:local_global} hold. 
In particular, in light of Example~\ref{ex:local_global_converse}, this asks for a classification of those matroid subdivisions which are compatible with a prescribed chirotope. 
For example, in the very recent works on positive Dressians and positive tropical Grassmannians \cite{AHrkaniLamSpradlin:2020, LukowskiParisiWilliams:2020, SpeyerWilliams:2020}, they consider the following special case of the problem:
Assign $+$ to all (increasingly ordered) bases of some uniform matroid (it is known that such a sign map is a chirotope, cf. Example~\ref{ex:alternating+matroid}).
Given a matroid subdivision of that matroid, when does the restriction of this chirotope to every piece also form a chirotope?
These works mostly consider regular matroid subdivisions, but in \cite{SpeyerWilliams:2020} they also constructed a non-regular matroid subdivision of $U_{3,12}$  by transversal matroids with such property.

Using a similar approach as in the proof of Theorem~\ref{thm:local_global}, we provide a criterion for deciding if, given a chirotope and a matroid subdivision of some matroid, its converse holds.

Let $\rchi$ be a chirotope supported on $U_{2,4}$.
There are exactly two terms among $\rchi(1,2)\rchi(3,4), -\rchi(1,3)\rchi(2,4), \rchi(1,4)\rchi(2,3)$ having the same sign, hence exactly one matroid subdivision of $U_{2,4}$ (necessarily non-trivial) does not satisfy the converse of Theorem~\ref{thm:local_global} (see Example~\ref{ex:local_global_converse}).
We call such a subdivision {\em forbidden} with respect to~$\rchi$.

\begin{proposition} \label{prop:local_global_converse}
Let $\rchi$ be a chirotope supported on some matroid $M$ and let $M_1,\ldots, M_k$ be a matroid subdivision of $M$.
The restrictions $\rchi_i$'s of $\rchi$ to $M_i$'s are all chirotopes if and only if the restriction of the matroid subdivision to every octahedron face is not forbidden.
\end{proposition}

\begin{proof}
  Suppose the restriction of the matroid subdivision to an octahedron face is forbidden with (3-dimensional) cells $C, C'$.
  Pick a full-dimensional cell $M_i$ that contains $C$, then the 3-term GP relation corresponding to the vertices of $C$ is violated in $M_i$.

Conversely, suppose some 3-term GP relation (\ref{eq:3GP}) is violated in $M_i$; denote by $F$ and $F'$ the faces of (the matroid polytopes of) $M$ and $M_i$ that contain bases involved in (\ref{eq:3GP}), respectively.
Since $\rchi$ is a chirotope but $\rchi_i$ is not, some bases of $M$ involved in (\ref{eq:3GP}) must be non-bases of $M_i$, so we have $F'\subsetneq F$.
Suppose $F$ is an octahedron.
$F'$ is either a pyramid or a square, and $F$ must have been subdivided in the matroid subdivision.
If $F'$ is a pyramid, then the subdivision of $F$ is by definition forbidden; if $F'$ is a square, pick any pyramid of the subdivision and the 3-term GP relation is still violated there, so the subdivision is also forbidden.
Now suppose $F$ is a pyramid.
$F'$ must be a square, but in such case $F$ would have violated the 3-term GP relation as $F'$ does.
\end{proof}

Proposition~\ref{prop:local_global_converse} could be quite hard to check algorithmically as there could be exponentially many octahedron faces, but if $\rchi$ has a nice description, one might expect to understand the situation better by analyzing the relative position of the octahedron faces together with their forbidden subdivisions.
For example, \cite[Theorem~4.3]{SpeyerWilliams:2020}  is a special case of our criterion.

\subsection{Duality} \label{sec:duality+chirotopes}

Another important concept for oriented matroids is \emph{duality}.
Recall from Example~\ref{ex:realizable+oriented+matroid} that an oriented matroid arising from a rectangular matrix can be represented in the form $(I_{d,d}|B)$, where we let $B$ be a $d \times n$ matrix for now.
Following~\cite[\S 7.2.5]{Handbook+DCG:2004}, its dual oriented matroid is represented by $(-\transpose{B}|I_{n \times n})$.
We mimic this by fixing a triangulations $\mathcal{T}$ of $\ssimplex_{d-1} \times \ssimplex_{n-1}$ as well as a sign matrix $A \in \{-1,+1\}^{d \times n}$, and constructing a pair of pointed polyhedral matching fields that gives a dual pair of oriented matroids.
This illustrates an advantage of considering pointed matching fields: they allow a more satisfactory duality theory.
\footnote{The duality for tropical oriented matroids in the sense of~\cite{ArdilaDevelin:2009} is obtained by essentially swapping the two node sets of size $d$ and $n$ of the bipartite graphs.
  This seems to be in contrast to the duality of classical oriented matroids, where both share the same ground set and have complementary ranks.
  However, it is resolved nicely by observing our construction below.
  In the duality of tropical oriented matroids, it is also crucial which part of the node set can have isolated nodes, as the two node sets have different roles. 
  Not only the two node sets but also this restriction is moved from one to the other side.
  But by passing to the pointed matching fields of size $(d,d+n)$ and $(n,n+d)$ as we do below, one sees that the isolated nodes implicitly give rise to a common ground set of size $d+n$. 
}
%% \footnote{In some earlier works, the dual of a $(d,n)$-matching field is defined as the $(n,d)$-matching field given by swapping the role of $\rows$ and $\ground$. While the construction works for any matching fields, it does not behave like the classical (oriented) matroidal duality.}.
%% \todo{GL: I will reformulate the footnote with the following aspects: classical dual OMs share the same ground set and have complementary ranks. The duality for tropical oriented matroids in the sense of Ardila \& Develin is obtained by swapping the two node sets of size $d$ and $n$ of the bipartite graphs. However, note that one also changes, on which side one allows for isolated nodes. By passing to the pointed extensions of the matching fields this condition of isolated nodes is turned into a complementary rank condition on a ground set of size $d+n$. }

The paragraph on ``Duality'' in~\cite[\S 3.5]{BLSWZ:1993} gives the following condition for two chirotopes to be dual.

\begin{lemma} \label{lem:duality+chirotopes}
  Let $\ground$ be a ground set of size $d + n$. 
  Let $\rchi \colon \ground^d \rightarrow\{+,-,0\}$ and $\rchi' \colon \ground^n\rightarrow\{+,-,0\}$ be two chirotopes of a uniform oriented matroid of rank $d$ and $n$, respectively. 
  They are dual if and only if for each ordering $(x_1,\dots,x_{d+n})$ of the ground set $\ground$ we have
  \[
  \rchi(x_1,\dots,x_d) \cdot \rchi'(x_{d+1},\dots,x_{d+n}) = \sign(x_1,\dots,x_{d+n}) ,
  \]
  where the latter sign is the sign of the permutation represented by the ordering of the $x_k$. 
\end{lemma}

As mentioned for the definition of a pointed polyhedral matching field (Definition~\ref{def:pointed+polyhedral}), one has to take the information of \emph{all} matchings occurring in the trees of $\mathcal{T}$ into account.
This corresponds to looking at all minors of $B$.
Equally, one can associate two different pointed polyhedral matching fields with $\mathcal{T}$ by swapping the factors of the product of simplices.
This leads to a pointed polyhedral matching field $(M_{\sigma})$ of size $(d,n+d)$ and $(N_{\sigma})$ of size $(n,d+n)$. 

$(M_{\sigma})$ and  $(N_{\sigma})$ are related combinatorially as follows.
Let $\mu$ be a matching on $\rows \sqcup \ground$.
Augment $\rows$ by another copy $\overline{\ground}$ of $\ground$ and $\ground$ by another copy $\overline{\rows}$ of $\rows$.
This leads to the two nodes sets $U = \rows \cup \overline{\ground}$ and $V = \overline{\rows} \cup \ground$, each union in this order.
Next, introduce another copy $\overline{\mu}$ between $\overline{\ground}$ and $\overline{\rows}$ by adding the edge $(\overline{j},\overline{i})$ for each edge $(i,j)$ in $\mu$.
Then, connect each isolated node to its copy.
We obtain a perfect matching $\tau$ on $U \sqcup V$, which can be thought as a permutation in $S_{d+n}$.
The two restricted matchings $\tau_{|\rows}$ and $\tau_{|\overline{\ground}}$ form a dual pair of matchings in $(M_{\sigma})$ and $(N_{\sigma})$, respectively.
The terminology is justified as they have complementary supports in $V$.
The notation is visualized in Figure~\ref{fig:extended+permutations+matchings}.

Finally, we let $\rchi_1$ be the chirotope induced by the matching field $(M_{\sigma})$ on $\rows \sqcup V = \rows \sqcup (\overline{\rows} \cup \ground)$ with the sign matrix $(I_{d,d} | A)$ and let $\rchi_2$ be the chirotope induced by the matching field $(N_{\sigma})$ on $\overline{\ground} \sqcup V = \overline{\ground} \sqcup (\overline{\rows} \cup \ground)$ with the sign matrix $(-\transpose{A} | I_{n,n})$.
This can also be interpreted as the sign matrix 
\[
H = 
\begin{pmatrix}
  I_{d,d} & A \\
  -\transpose{A} & I_{n,n}
\end{pmatrix}
\]
on the bipartite graph $K_{d+n,d+n}$.

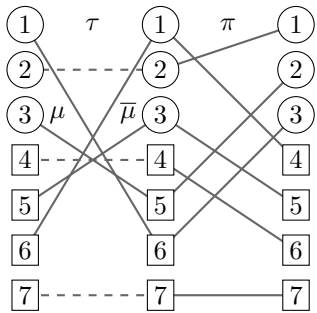
\begin{figure}
  \centering
\newcommand{\bigraphsevencoord}[5]{
  \coordinate (u0) at (#1,#2+3*#3);
  \coordinate (u1) at (#1,#2+2*#3);
  \coordinate (u2) at (#1,#2+#3);
  \coordinate (u3) at (#1,#2);
  \coordinate (u4) at (#1,#2-#3);
  \coordinate (u5) at (#1,#2-2*#3);
  \coordinate (u6) at (#1,#2-3*#3);

    \coordinate (v0) at (#4,#2+3*#3);
  \coordinate (v1) at (#4,#2+2*#3);
  \coordinate (v2) at (#4,#2+#3);
  \coordinate (v3) at (#4,#2);
  \coordinate (v4) at (#4,#2-#3);
  \coordinate (v5) at (#4,#2-2*#3);
  \coordinate (v6) at (#4,#2-3*#3);

    \coordinate (w0) at (#5,#2+3*#3);
  \coordinate (w1) at (#5,#2+2*#3);
  \coordinate (w2) at (#5,#2+#3);
  \coordinate (w3) at (#5,#2);
  \coordinate (w4) at (#5,#2-#3);
  \coordinate (w5) at (#5,#2-2*#3);
  \coordinate (w6) at (#5,#2-3*#3);
}
\newcommand{\bigraphsevennodes}{
\node[VertexStyle]  at (u0){1};
\node[VertexStyle]  at (u1){2};
\node[VertexStyle]  at (u2){3};
\node[BoxVertex]  at (u3){4};
\node[BoxVertex]  at (u4){5};
\node[BoxVertex]  at (u5){6};
\node[BoxVertex]  at (u6){7};

\node[VertexStyle]  at (v0){1};
\node[VertexStyle]  at (v1){2};
\node[VertexStyle]  at (v2){3};
\node[BoxVertex]  at (v3){4};
\node[BoxVertex]  at (v4){5};
\node[BoxVertex]  at (v5){6};
\node[BoxVertex]  at (v6){7};

\node[VertexStyle]  at (w0){1};
\node[VertexStyle]  at (w1){2};
\node[VertexStyle]  at (w2){3};
\node[BoxVertex]  at (w3){4};
\node[BoxVertex]  at (w4){5};
\node[BoxVertex]  at (w5){6};
\node[BoxVertex]  at (w6){7};
}

\begin{tikzpicture}[scale=0.3]

  \bigraphsevencoord{0}{2}{2}{6}{12}

  \draw[EdgeStyle] (u0) to (v5);
  \draw[EdgeStyle] (u2) to (v4);
  \draw[EdgeStyle] (u4) to (v2);
  \draw[EdgeStyle] (u5) to (v0);

  \draw[EdgeStyle,dashed] (u1) to (v1);
  \draw[EdgeStyle,dashed] (u3) to (v3);
  \draw[EdgeStyle,dashed] (u6) to (v6);

  \draw[EdgeStyle] (v0) to (w3);
  \draw[EdgeStyle] (v1) to (w0);
  \draw[EdgeStyle] (v2) to (w4);
  \draw[EdgeStyle] (v3) to (w5);
  \draw[EdgeStyle] (v4) to (w1);
  \draw[EdgeStyle] (v5) to (w2);
  \draw[EdgeStyle] (v6) to (w6);
  
  \bigraphsevennodes

  \node at (3,8){$\tau$};
  \node at (9,8){$\pi$};

  \node at (3,4.1){$\mu$ \hspace{1.4em} $\overline{\mu}$};
%  \node at (4.5,4){$\overline{\mu}$};
  
\end{tikzpicture}
  \caption{The matchings arising in the proof of Proposition~\ref{prop:duality+chirotopes}. In this picture, we have $|\rows| = 3$ and $|\ground| = 4$. }
  \label{fig:extended+permutations+matchings}
\end{figure}

\begin{proposition} \label{prop:duality+chirotopes}
  The two chirotopes $\rchi_1$ and $\rchi_2$ describe a dual pair of oriented matroids. 
\end{proposition}
\begin{proof}
  Let an arbitrary partition of $V\cong[d+n]$ into two ordered sets $P_1 := \{x_1 < \dots < x_d\}$ and $P_2 := \{x_{d+1} < \dots < x_{d+n}\}$ be given.
  This defines the bijection $\pi$ in $S_{d+n}$ which maps $(x_1,\dots,x_d,x_{d+1},\dots,x_{d+n})$ to $(1,2,\dots,d+n)$. 

  Let $\mu_1$ be the matching in $(M_{\sigma})$ on $P_1$ and let $\mu_2$ be the matching in $(N_{\sigma})$ on~$P_2$.
  By construction, their union is a perfect matching $\tau$ on $U\sqcup V$, coming from a matching $\mu$ on $\rows \sqcup \ground$. 
  Using Definition~\ref{def:sign+of+a+matching} and~\eqref{eq:sign+map+matching+field}, we obtain
  \begin{align*}
    \rchi_1(P_1) = \sign(\pi\cdot\tau_{|\rows})\prod_{e\in \tau_{|\rows}}H_e \quad \text{ and } \quad \rchi_2(P_2) = \sign(\pi\cdot\tau_{|\overline{\ground}})\prod_{e\in \tau_{|\overline{\ground}}}(-H_e) \;.
  \end{align*}
  As the concatenation $\pi \cdot \tau$ of the matchings $\tau$ and $\pi$ is disjoint on $\rows$ and $\overline{\ground}$, we obtain
  \[
  \rchi_1(P_1)\cdot \rchi_2(P_2) = \sign(\pi) \sign(\tau)\prod_{e\in \tau_{|\overline{\ground}}}(-H_e)\prod_{e\in \tau_{|\rows}}H_e \,.
  \]
  Using the definition of $\tau$ based on $\mu$, we get
  \[
  \prod_{e\in \tau_{|\overline{\ground}}}(-H_e)\prod_{e\in \tau_{|\rows}}H_e = \prod_{e\in \mu}(-H_e)\prod_{e\in \mu}H_e = (-1)^{|\mu|} .
  \]
Now we compute the sign of $\tau$ by counting inversions.
Every inversion within $\mu$ is paired up with an inversion within $\overline{\mu}$ and vice versa, while there are $|\mu|^2$ ($\equiv |\mu| \mod{2}$) many inversions between $\mu$ and $\overline{\mu}$, so $\sign(\tau) = (-1)^{|\mu|}$.
 Since $\pi$ is just the inverse of the permutation encoded by $(x_1,\dots,x_d,x_{d+1},\dots,x_{d+n})$, the proposition follows from Lemma~\ref{lem:duality+chirotopes}.
\end{proof}

\subsection{Signed Circuits and Cocircuits} \label{sec:signed+co+circuits}

Knowing that $\rchi$ is a chirotope, one can use the equivalence of axiom systems as described in \cite[Chapter 3]{BLSWZ:1993} to convert $\rchi$ into other objects associated to the oriented matroid. 
However, it is also possible to construct many of these objects more directly using graphs in the theory of matching fields.
We shall describe the construction of signed circuits and cocircuits here, while Section~\ref{sec:topes+covectors} will study covectors in general.
In particular, linkage pd-graphs form the analogue of signed circuits and Chow pd-graphs form the analogue of signed cocircuits.
Observe that a matching field can be recovered from linkage pd-graphs and equally well from Chow pd-graphs, see~\cite{LohoSmith:2020}, in analogy to the cryptomorphic description of an oriented matroid by signed circuits or signed cocircuits. 
% \todo[inline]{MC: Have we defined Chow pd-graphs yet? There is a description as mininal transversals in the introduction but I am not sure that meaning is clear.} 
Let $\tau\subset \ground$ be a $(d+1)$-subset.
Recall from Example~\ref{ex:linkage+matching+field} that the linkage pd-graph $T_{\tau}:=\bigcup_{\sigma\subset\tau} M_\sigma$ is a spanning tree of $K_{\rows,\tau}$ in which every node in $\rows$ has degree 2.
One can then define an auxiliary tree $\widetilde{T}_\tau$ whose node set is $\tau$ and two nodes are connected by an edge if they are both incident to the same $r\in \rows$ in $T_{\tau}$. 
Note that this is the same as identifying a node $j$ in $\tau$ with the matching on $\rows \sqcup (\tau \setminus \{j\})$ and connecting two matchings if they only differ by one edge. 
Such a tree $\widetilde{T}_\tau$, together with an edge labeling using $\rows$, is a {\em linkage tree} in~\cite{SturmfelsZelevinsky:1993}.  
We equip $\widetilde{T}_\tau$ with a sign labeling in the following way.
To the edge $(u,v) \in \widetilde{T}_\tau$ associated to the two edges $(u,r),(v,r)\in T_\tau$, we assign the negative of the product of the signs on $(u,r)$ and $(v,r)$. 

\begin{proposition} \label{prop:construction+signed+circuits}
For any $\tau$, there exist precisely the two sign assignments of $\tau$ such that two adjacent nodes in $\widetilde{T}_\tau$ are of the same sign if and only if the edge between them is positive. 
These are the two signed circuits supported on $\tau$. 
\end{proposition}

\begin{proof}
Pick an arbitrary root for $\widetilde{T}_\tau$ and choose a sign for it, then the sign of every other node is uniquely determined by its parent. 
We show that a signed circuit with respect to $\rchi$ satisfies the constraint described in the statement.

Let $C$ be a signed circuit supported on $\tau$. Write $\rows=\{r_1<\ldots<r_d\},\tau=\{e_1<\ldots<e_{d+1}\}$ and let $x=e_i, y=e_j$ be any two nodes. 
By \cite[Theorem~3.5.5]{BLSWZ:1993}, 
$$C(e_j)=-C(e_i)\rchi(e_i,e_2,\ldots,\widehat{e_i},\ldots,\widehat{e_j},\ldots,e_{d+1})\rchi(e_j,e_2,\ldots,\widehat{e_i},\ldots,\widehat{e_j},\ldots,e_{d+1}).$$
By permuting elements in $\rows$ and $\tau$ if necessary, we may assume $i=1,j=2$, and $M_{\tau\setminus\{e_1\}}$ matches each $r_k$ to $l_{k+1}$ . 
We first assume that $x,y$ are adjacent in $\widetilde{T}_\tau$, in particular, $M_{\tau\setminus\{e_2\}}=M_{\tau\setminus\{e_1\}}\setminus\{r_1e_2\}\cup\{r_1e_1\}$.
If the edges $r_1e_1$ and $r_1e_2$ have the same sign, then $-\rchi(e_1,e_3,\ldots,e_{d+1})\rchi(e_2,\ldots,e_{d+1})=-\sign(M_{\tau\setminus\{e_2\}})\sign(M_{\tau\setminus\{e_1\}})=-1$ as the two chirotopes have the same products of signs of edges.
Now $-1=-A_{r_1e_1}A_{r_1e_2}$ is what we have labeled the edge between $x,y$ in $\widetilde{T}_\tau$.
The opposite case when $r_1e_1$ and $r_1e_2$ having opposite signs is similar, and we can extend the comparison of signs to any pair of nodes by induction.
\end{proof}

Now we consider signed cocircuits of our oriented matroid.
Recall that cocircuits correspond to vertices in an arrangement encoding an oriented matroid.
These appear in the framework of abstract tropical linear programming~\cite{Loho:2016} as certain trees with prescribed degree sequence, see also Remark~\ref{rem:connection+tropical+simplex}. 
The statement of \cite[Theorem~12.9]{Postnikov:2009} identifies the correct tree for our purpose.

\begin{proposition} \label{prop:chow+tree}
For any $(n-d+1)$-subset $\rho$, there exists a unique spanning tree $T_\rho$, encoding a cell in the triangulation, such that every node in $\rho$ is of degree 1 and every node in $\ground\setminus\rho$ is of degree $2$.
\end{proposition}

We say we {\em flip} a node in $\rows$ if we negate the row of $A$ indexed by the node.

\begin{proposition}\label{prop:cocircuit}
For any $\rho$, there exist precisely two flippings of $\rows$ such that every node of $\ground\setminus\rho$ is incident to edges of different signs in $T_\rho$. 
In each case, the signs of the edges incident to $\rho$ together give a signed cocircuit supported on $\rho$.
\end{proposition}

It is possible to give a combinatorial proof in the same vein of Proposition~\ref{prop:construction+signed+circuits}, but as cocircuits are special instances of covectors, we apply the result in Section~\ref{sec:topes+covectors} (with no circular argument).
In particular, we give a more formal definition of flips there.
We put a separate discussion here because of the extra structural properties of cocircuits.

\begin{proof}
  We construct an auxiliary tree $\widetilde{T}_\rho$ whose node set is $\rows$.
  Two nodes are connected by an edge if they are both incident to some node in $\ground\setminus\rho$ in $T_\rho$ and we label the edge by the said node in $\ground\setminus\rho$.
Pick an arbitrary root for $\widetilde{T}_\rho$ and choose a flipping decision for it, then the flipping decision of every other node is uniquely determined by its parent.
Apply Corollary~\ref{coro:more+general+covectors} to $T_\rho$ and any of the two flipping decisions, we know that the signs of the edges incident to $\rho$ is a covector of the oriented matroid supported on $\rho$, i.e., it is a cocircuit.
\end{proof}

\begin{remark}
  Perhaps a more intuitive way to understand Proposition~\ref{prop:cocircuit} is to use the geometric picture in the follow-up of our paper (see Section~\ref{sec:ringel} for a rank 3 example).
  The mixed cell representing $T_\rho$, reflected to the correct orthant in the patchworking complex specified by the flipping, is dual to the intersection of $d-1$ pseudohyperplanes.
  Thus it recovers the intuition from classical hyperplane arrangements.
\end{remark}

\begin{example}
  We illustrate the construction of the signed circuits and cocircuits by continuing Example~\ref{ex:diagonal+2+4+signs}.
  The left graph in Figure~\ref{fig:covectors+with+signs} shows the linkage pd-graph $T_{\tau}$ for $\tau = \{1,2,3\}$ used in Proposition~\ref{prop:construction+signed+circuits}.
  The signs on the nodes in $\rows$ are chosen negative so that one can think of multiplying the signs along paths instead of forming a signed linkage tree as constructed before Proposition~\ref{prop:construction+signed+circuits}. 

  The right graph in Figure~\ref{fig:covectors+with+signs} depicts the spanning tree $T_{\rho}$ for $\rho = \{1,2,4\}$ as constructed in Proposition~\ref{prop:cocircuit}.
  The sign vector on $\rows$ corresponds to the flippings. 
  Note that the two sign vectors $(+,+)$ and $(-,-)$ on the nodes in $\rows$ correspond to the two orthants in which the (pseudo-)spheres corresponding to the elements in $\ground \setminus \rho$ intersect.  
\end{example}

\begin{figure}
  \centering
  \begin{tikzpicture}[scale=0.6]

\bigraphtwofourcoord{0}{0}{1.2}{0.6}{2.5};
\draw[EdgeStyle] (v1) to node[above=-0.1cm] {$+$} (w1);
\draw[EdgeStyle] (v2) to node[above=-0.1cm] {$-$} (w1);
\draw[EdgeStyle] (v2) to node[above=-0.1cm] {$-$} (w2);
\draw[EdgeStyle] (v3) to node[above=-0.1cm] {$-$} (w2);
\bigraphtwofournodes

\node[left=0.4cm of w1] {$-$};
\node[left=0.4cm of w2] {$-$};

\node[right=0.4cm of v1] {$+$};
\node[right=0.4cm of v2] {$+$};
\node[right=0.4cm of v3] {$-$};

%  \draw[bend left] (1) to node[above] {3} (2);

\bigraphtwofourcoord{7}{0}{1.2}{0.6}{2.5};
\draw[EdgeStyle] (v1) to node[above=-0.1cm] {$+$}(w1);
\draw[EdgeStyle] (v2) to node[above=-0.1cm] {$-$}(w1);
\draw[EdgeStyle,red] (v3) to node[above=-0.1cm] {$+$}(w1);
\draw[EdgeStyle,red] (v3) to node[above=-0.1cm] {$-$}(w2);
\draw[EdgeStyle] (v4) to node[above=-0.1cm] {$+$} (w2);
\bigraphtwofournodes

\node[left=0.4cm of w1] {$+$};
\node[left=0.4cm of w2] {$+$};

\node[right=0.4cm of v1] {$+$};
\node[right=0.4cm of v2] {$-$};
\node[right=0.4cm of v4] {$+$};

\end{tikzpicture}
  \caption{Two pd-graphs corresponding to a signed circuit and a signed cocircuit. }
   \label{fig:covectors+with+signs}
\end{figure}
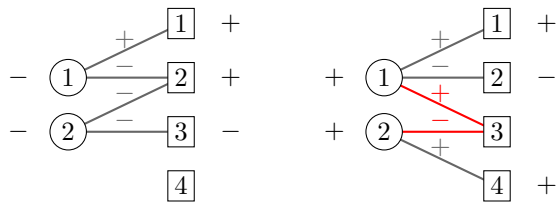

We give a further interpretation of the trees from Proposition~\ref{prop:chow+tree} in terms of {\em Chow pd-graphs}.
Recall that the latter is a minimal subset of the edges of $K_{d,n}$ intersecting all matchings. 
These were introduced in \cite[Eqn.~5.1]{SturmfelsZelevinsky:1993} and they were called `Chow covectors' in~\cite{LohoSmith:2020}. 

\begin{proposition}
The collection of edges $\Omega_\rho$ incident to $\rho$ in $T_\rho$ is the Chow pd-graph supported on $\rho$. 
\end{proposition}

\begin{proof}
We claim that there is a (unique) perfect matching $\widetilde{M}_{\rows'}$ between any $(d-1)$-subset $\rows'\subset \rows$ and $\ground\setminus\rho$ in $T_\rho$:
set the missing node of $\rows'$ as the root of $\widetilde{T}_\rho$ in the proof of Proposition~\ref{prop:cocircuit}, and match every other node with (the label of) the edge towards its parent.
From this, for any $x\in\rho$ with adjacent node $r$, $\widetilde{M}_{\rows\setminus r}\cup\{xr\}$ is a perfect matching between $\rows$ and $\ground\setminus\rho\cup\{x\}$.
This gives rise to the matchings required in the definition~\cite[Eqn.~5.1]{SturmfelsZelevinsky:1993} of a Chow pd-graph. 
\end{proof}

The next corollary follows from \cite[Theorem~1]{BernsteinZelevinsky:1993}, and is a matching field analogue of the matroidal fact that cocircuits are precisely the minimal transversals of bases.

\begin{corollary} \label{coro:cocircuit_transversal}
  The minimal transversals of a polyhedral matching field are given by the leaves of $T_{\rho}$ in $\rho$ where $\rho$ ranges over all $(n-d+1)$-subsets. 
\end{corollary}

We also have a matching field analogue of the orthogonality of circuits and cocircuits, i.e., a circuit can not intersect a cocircuit by exactly one element.
Let $\tau, \rho\subset \ground$ be of size $d+1$ and $n-d+1$, respectively, and let $T_{\tau}$ and $T_{\rho}$ be the trees constructed for Proposition~\ref{prop:construction+signed+circuits} and Proposition~\ref{prop:cocircuit}, respectively.

\begin{proposition}
  $T_{\tau}$ contains at least two leafs of $T_{\rho}$.
\end{proposition}

\begin{proof}
Pick an arbitrary $\sigma\subset\tau$ of size $d$.
By Corollary~\ref{coro:cocircuit_transversal}, $M_\sigma\subset T_{\tau}$ contains at least an edge $re$ of $T_{\rho}$.
Now $M_{\tau\setminus\{e\}}\subset T_{\tau}\setminus\{re\}$ must contain at least another.
\end{proof}

\subsection{Topes and Covectors} \label{sec:topes+covectors}

In the development of tropical oriented matroids, the tropical analogue of \emph{topes} and \emph{covectors} played an important role.
We adapt the notion to our terminology and provide the connection with the oriented matroid $\orientedmatroid$ represented by the chirotope~$\rchi$. 

Fix a polyhedral matching field $(M_{\sigma})$ extracted from the special trees $\Ical$ of a triangulation $\mathcal{T}$ of $\ssimplex_{d-1}\times\ssimplex_{n-1}$, and a sign matrix $A\in\{-1,1\}^{\rows\times\ground}$.
A \emph{covector pd-graph} is a subgraph of a tree in $\mathcal{T}$ with no isolated node in $\ground$.
A \emph{tope pd-graph} is a covector pd-graph for which each node in $\ground$ has degree $1$; it is called \emph{inner tope pd-graph}, if it is actually a subgraph of a tree in $\Ical$.
Note that linkage pd-graphs and the trees identifying Chow pd-graphs of a polyhedral matching field are special cases of covector pd-graphs. 

%\todo[inline]{MC: For consistency, maybe we should call these covector pd-graphs, tope pd-graphs, etc.?}

\begin{definition} \label{def:signed+forest}
Given $S\in\{-1,0,1\}^{\rows}$ and $F\subseteq\rows\times\ground$, we define the sign matrix $SA_{F}\in\left\{ -1,0,1\right\} ^{\rows\times\ground}$
by
\[
(SA_{F})_{i,j}=\begin{cases}
S_{i}A_{i,j}, & (i,j)\in F,\\
0, & \text{otherwise.}
\end{cases}
\]
\end{definition}

\begin{definition} \label{def:triang+to+covectors+OM}
  Given a subgraph $F$ of $K_{\rows,\ground}$
  and a sign vector $S\in\{-1,0,1\}^{\rows}$, define the sign vector $\psi_A(S,F)= Z \in\{-1,0,1\}^{\ground}$, where 
\[
Z_{j}=
\begin{cases}
0, & \text{column \ensuremath{j} of \ensuremath{SA_{F}} has positive and negative entries, or all zeros}\\
1, & \text{column \ensuremath{j} of \ensuremath{SA_{F}} has only non-negative entries, and is non-zero}\\
-1, & \text{column \ensuremath{j} of \ensuremath{SA_{F}} has only non-positive entries, and is non-zero.}
\end{cases}
\]
\end{definition}

\begin{example} \label{ex:covector+tope+signs}
  The graphs in Figure~\ref{fig:reducing+covector+to+tope} together with the sign vector $S = (0,-1,1)$ and sign matrix $\widetilde{A} = (I_{3,3} | A)$ with 
  \[
  A =
  \begin{pmatrix}
    1 & 1 & -1 \\
    -1 & -1 & 1 \\
    -1 & -1 & -1
  \end{pmatrix} 
  \]
  give rise to the matrices 
  \[
  F = 
  \begin{pmatrix}
    0 & 0 & 0 & 0 & 0 & 0 \\
    0 & -1 & 0 & 1 & 1 & 0 \\
    0 & 0 & 1 & -1 & 0 & -1
  \end{pmatrix}
  \quad
  \text{ and }
  \quad
  T = 
    \begin{pmatrix}
    -1 & 0 & 0 & 0 & 0 & 0 \\
    0 & -1 & 0 & 1 & 1 & 0 \\
    0 & 0 & 1 & 0 & 0 & -1
  \end{pmatrix} . 
    \]
    Their images are
    \[
    X = \psi_{\widetilde{A}}(S,F) = (0,-1,1,0,1,-1) \text{ and } Y = \psi_{\widetilde{A}}(S,F) = (-1,-1,1,1,1,-1) .
    \]
    These sign vectors fulfill $Y \geq X$ and highlight the construction used in the proof of Proposition~\ref{prop:covectors+from+topes}. 
\end{example}

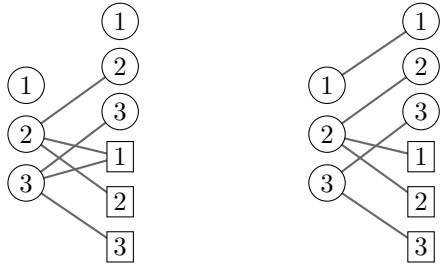
\begin{figure}[htb]
  \centering
\begin{tikzpicture}[scale=0.5]

  \bigraphthreesixcoord{0}{0}{1.2}{1.3}{2.5}
  \draw[EdgeStyle] (w2) to (v2);
  \draw[EdgeStyle] (w2) to (v4);
  \draw[EdgeStyle] (w2) to (v5);
  \draw[EdgeStyle] (w3) to (v3);
  \draw[EdgeStyle] (w3) to (v4);
  \draw[EdgeStyle] (w3) to (v6);
  \bigraphthreesixnodes

  \bigraphthreesixcoord{8}{0}{1.2}{1.3}{2.5}
  \draw[EdgeStyle] (w1) to (v1);
  \draw[EdgeStyle] (w2) to (v4);
  \draw[EdgeStyle] (w2) to (v2);
%  \draw[EdgeStyle] (w2) to (v4);
  \draw[EdgeStyle] (w2) to (v5);
  \draw[EdgeStyle] (w3) to (v3);
  \draw[EdgeStyle] (w3) to (v6);
  \bigraphthreesixnodes

\end{tikzpicture}

  \caption{On the left is a graph on $\rows\sqcup\widetilde{\ground}$ with $|\rows| = |\ground| = 3$ and on the right is a related inner tope pd-graph discussed in Example~\ref{ex:covector+tope+signs}. }
  \label{fig:reducing+covector+to+tope}
\end{figure}

For an oriented matroid $\orientedmatroid$, let $\mathcal{V}^*(\orientedmatroid)$ denote the poset of covectors of $\orientedmatroid$. 
There is a \emph{weak map} $\orientedmatroid\rightsquigarrow\omtwo$ between two oriented matroids $\orientedmatroid,\omtwo$ on $\ground$, if for all covectors $X$ of $\omtwo$, there exists a covector $Y$ of $\orientedmatroid$ such that $Y\geq X$.
\footnote{See \cite[Theorem 0.1]{Anderson:2001} for a more modern definition: there is a weak map if there is a surjective poset map $g:\mathcal{V}^*(\orientedmatroid)\rightarrow \mathcal{V}^*(\omtwo)$ such that $g(X)\leq X$ for all $X\in\mathcal{V}^*(\orientedmatroid)$. }

If $\orientedmatroid$ and $\omtwo$ have the same rank, then this is the same as saying that, up to a global sign change, $\rchi_{\orientedmatroid}\geq\rchi_{\omtwo}$ by \cite[Proposition 7.7.5]{BLSWZ:1993}.
In particular, if $\rchi_{\omtwo}$ is a chirotope obtained by restricting $\rchi_{\orientedmatroid}$ to a cell in some matroid subdivision  of $\orientedmatroid$, such as in Lemma \ref{lem:TM_real}, then $\orientedmatroid\rightsquigarrow\omtwo$ is a weak map.

\begin{proposition} \label{prop:tope+from+tope+graph}
  For each sign vector $S\in\{-1,1\}^{\rows}$ and each inner tope pd-graph~$F$, $\psi_A(S,F)$ is a tope of the oriented matroid given by the chirotope $\rchi$.
\end{proposition}
\begin{proof}
Pick a tree $T \in \Ical$ that contains $F$ and denote by $\omtwo$ the oriented matroid representing $\rchi_T$.
By Lemma~\ref{lem:TM_real}, $\omtwo$ is realized by any matrix of the form $A(T)$, in which we choose $\pm 1$ for entries in $T$ but not $F$, and $\pm 2d$ for entries in $F$.
This ensures that the sign pattern of the row space element $\transpose{S} \cdot A(T)$ equals $\psi_A(S,F)$, hence $\psi_A(S,F)$ is a tope of $\omtwo$.
From the discussion above, there is a weak map $\orientedmatroid\rightsquigarrow\omtwo$ , so $\psi_A(S,F)$ is a tope of $\orientedmatroid$ as well.
\end{proof}

To extend this correspondence to more general covectors, we make use of a result by Mandel.

\begin{theorem}[{\cite[Thm.~4.2.13]{BLSWZ:1993}}]
  Let $\Topes$ be the set of topes of the oriented matroid~$\orientedmatroid$.
  Then $X \in \{-1,0,1\}^{\ground}$ is a covector of $\orientedmatroid$ exactly if $X \circ \Topes \subseteq \Topes$. 
\end{theorem}

Recall that for $X,Y \in \{-1,0,1\}^{\ground}$, the composition $X \circ Y$ agrees with $X$ in all positions $e \in \ground$ with $X_e \neq 0$, and agrees with $Y$ otherwise.

\begin{corollary} \label{coro:covectors+defined+by+topes}
  Let $X\in\{-1,0,1\}^{\ground}$ be a sign vector such that every $Y\in\{-1,1\}^{\ground}$ satisfying $Y\geq X$ is a tope of $\orientedmatroid$. Then $X$ is a covector of $\orientedmatroid$.
\end{corollary}

We now consider the pointed polyhedral matching field $\widetilde{(M_{\sigma})}$ encoding the starting triangulation, as well as the sign matrix $\widetilde{A} := (I_{d,d} | A)$. 
By Theorem~\ref{thm:main}, they induce an oriented matroid $\widetilde{\orientedmatroid}$.
Let $F$ be a subgraph of a tree in $\widetilde{\mathcal{T}}$ without isolated nodes in $\ground \subset \widetilde{\ground}$ and such that a node in $\widetilde{\rows} \subset \widetilde{\ground}$ is isolated only if the corresponding node in $\rows$ is isolated as well. 
Furthermore, let $S \in \{-1,0,1\}^{\rows}$ be a sign vector whose support contains the set of non-isolated nodes of $F$ in $\rows$.

\begin{proposition} \label{prop:covectors+from+topes}
  The sign vector $\psi_{\widetilde{A}}(S,F)$ is a covector of $\widetilde{\orientedmatroid}$. 
\end{proposition}
\begin{proof}
  By Corollary~\ref{coro:covectors+defined+by+topes}, it suffices to show that every full sign vector $Y\geq \psi_{\widetilde{A}}(S,F)$ is a tope of $\widetilde{\orientedmatroid}$.
  This amounts to finding a full sign vector $S'$ and an inner tope pd-graph $T$ such that $Y=\psi_{\widetilde{A}}(S',T)$ by Proposition~\ref{prop:tope+from+tope+graph}, since the extended trees in $\widetilde{\mathcal{T}}$ have degree at least $2$ everywhere in $\rows$. 
We first construct $S'$ by setting each zero coordinate $i\in\rows$ of $S$ to $Y_i$, and we include those missing edges between $\rows$ and its copy $\widetilde{\rows}$.
Now for each coordinate $j\in\ground$ of $\psi_{\widetilde{A}}(S,F)$, the $j$-th column of $S\widetilde{A}_{F}$ must contain an entry whose sign equals $Y_j$, so we remove all edges incident to $j$ except the one corresponding to that entry.
This results in an inner tope pd-graph $T$, and from our construction it is easy to see that $Y=\psi_{\widetilde{A}}(S',T)$ as desired.
\end{proof}

Restricting the covectors to $\ground$ yields the following. 

\begin{corollary} \label{coro:more+general+covectors}
  For a covector pd-graph $F$, $\psi_A(S,F)$ is a covector of $\orientedmatroid$ for every sign vector $S$.
\end{corollary}

We shall see in the follow-up of our paper that the map is actually surjective.
The strategy of understanding the oriented matroid via ``local pieces'' makes it difficult to combinatorially capture all covectors in the oriented matroid.
Instead, our proof uses topological method, particularly the Borsuk--Ulam theorem together with the Topological Representation Theorem.

\begin{remark} \label{rem:connection+tropical+simplex}
  Now that we can transfer the covector pd-graphs to covectors of the oriented matroid $\orientedmatroid$, we give a brief overview on the connection with simplex-like algorithms for tropical linear programming.

  The iteration in the framework of abstract tropical linear programming~\cite{Loho:2016} is described in the language of trees encoding a triangulation.
  The main object in each iteration is a \emph{basic covector}, which represents the analogue of a basic point in the simplex method.
  These are formed by a subclass of those trees from Proposition~\ref{prop:chow+tree} and, from our new perspective, give rise to certain cocircuits of~$\orientedmatroid$. 

  One can go even further and consider the tropicalization of the simplex method~\cite{ABGJ-Simplex:A}.
  Using their genericity assumption, one sees that the covector pd-graph of the basic point defined in~\cite[Proposition-Definition~3.8]{ABGJ-Simplex:A} is also such a tree.
  Their pivoting depends on the signs of the tropical reduced costs, which are deduced from a ``Cramer digraph''~\cite[\S 5]{ABGJ-Simplex:A}.
  The latter can also be considered as covector pd-graphs and hence give rise to covectors for $\rchi$. 
\end{remark}

%\section{On Oriented Matroids arising from Matching Fields}
\section{Realizability}
\label{sec:OM_MF}

\subsection{Sets of Oriented Matroids}

Starting from a polyhedral matching field $(M_\sigma)$,
we can associate an oriented matroid for each sign matrix $A \in \{+,-\}^{d \times n}$.
More generally, we denote the sign map defined in~\eqref{eq:sign+map+matching+field} by $\rchi((M_{\sigma}),A)$.   
This leads to the set of sign maps
\[
\omm\left[(M_{\sigma})\right] \ =\ \{\rchi((M_{\sigma}),A) \colon A \in \{+,-\}^{d \times n}\} \enspace ,
\]
or just $\omm$ if the matching field is clear.
By the \emph{support} $\supp\left[(M_{\sigma})\right]$ of a matching field, we mean the union of all edges occurring in a matching. 
Clearly, only the signs in the entries corresponding to elements in the support of the matching field influence the resulting sign map. 

\begin{example} \label{ex:single+matching}
  A $(d,d)$-matching field consists of a single matching. Hence, among the $2^d$ sign matrices differing on the support of the matching field, we get $+$ for half of the matrices and $-$ for the other half. 
\end{example}

\begin{example} \label{ex:linkage+matching+field+covector}
  A linkage $(d,d+1)$-matching field $(M_{\sigma})$, see Example~\ref{ex:linkage+matching+field} and Section~\ref{sec:signed+co+circuits}, can be identified with its linkage pd-graph $T$.
  Its support has cardinality $d + (d+1) - 1 = 2d$.
  To determine all oriented matroids in $\omm$, it suffices to fix the signs on a matching.
  The remaining degrees of freedom correspond to the $d$ edges of the linkage tree $\widetilde{T}$.

  In particular, all $2^d$ assignments of $+$ or $-$ to the edges of $\widetilde{T}$ yield different sign maps:
  Let $k$ be the node in $\widetilde{T}$ corresponding to the fixed matching and let $A$, $A'$ be different assigments.
  Then there is a path emerging from $k$ on which the signs assigned by $A$ and $A'$ differ.
  Let $(u,v)$ be the edge closest to $k$ on this path, where $u$ denotes the node closer to $k$; observe that $u$ could be equal to $k$. 
  Then the signs $\rchi((M_{\sigma}),A)(\sigma_v)$ and $\rchi((M_{\sigma}),A')(\sigma_v)$ differ, where $\sigma_v$ is the subset of $\ground$ corresponding to the node $v$ in $\widetilde{T}$.
\end{example}

We summarize from Theorem~\ref{thm:main}.

\begin{corollary}
  If $(M_{\sigma})$ is a polyhedral matching field, then all sign maps in $\omm\left[(M_{\sigma})\right]$ are chirotopes. 
\end{corollary}

For the subclass of coherent matching fields, see Example~\ref{ex:coherent+matching+field}, we know that the matchings are given as weight maximal matchings induced by a weight matrix $M = (m_{ij})_{(i,j) \in [d] \times [n]} \in \RR^{d \times n}$.
As a commonly used trick in tropical geometry, we consider the matrix $(t^M) := (t^{m_{ij}})_{(i,j) \in [d] \times [n]} \in \RR^{d \times n}$ for a sufficiently large parameter $t > 0$.
Then the maximal term in the expansion of the determinant corresponds to the weight maximal matching of the coherent matching field.

\begin{corollary}
  If $(M_{\sigma})$ is a coherent matching field, then all sign maps in $\omm\left[(M_{\sigma})\right]$ are chirotopes of oriented matroids realizable over $\RR$. 
\end{corollary}

\begin{example} \label{ex:alternating+matroid}
  For the diagonal $(d,n)$-matching field, we obtain all reorientations of the unique uniform positroid, the alternating matroid; see also~\cite[\S 8.2]{BLSWZ:1993}. 
\end{example}

The observation of Example~\ref{ex:linkage+matching+field+covector} can be extended to arbitrary linkage matching fields.
Note that the sign map does not have to be a chirotope if the matching field is not polyhedral. 
We fix a linkage matching field $(M_{\sigma})$. 
By~\cite[Thm.~3.2]{SturmfelsZelevinsky:1993}, each node of $\rows$ in the support of a linkage matching field, considered as a subgraph of $K_{d,n}$, has $n-d+1$ neighbours.
In particular, the support has cardinality $d \times (n-d+1)$.
We fix a $d$-subset $\sigma_0$ of $\ground$.
Let $A$ and $A'$ be two different sign matrices supported on the support $\supp\left[(M_{\sigma})\right]$ which agree on $\sigma_0$. 

\begin{proposition} \label{prop:distinct+sign+maps+different+matrices}
  The sign maps $\rchi((M_{\sigma}),A)$ and $\rchi((M_{\sigma}),A')$ are distinct. 
\end{proposition}
\begin{proof}
  Consider the flip graph of the linkage matching field as introduced in~\cite[\S3.5]{LohoSmith:2020}; that is the graph which has the matchings as nodes and two nodes are connected if the matchings differ by exactly one edge.
  Let $\mu_1$ and $\mu_2$ be two matchings such that $A$ and $A'$ differ on the matchings. 
  Since the linkage graphs cover the flip graph, there is a path from $\mu_1$ to $\mu_2$.
  As in Example~\ref{ex:linkage+matching+field+covector}, the subset of $\ground$ corresponding to the node in the flip graph, where the signs of $A$ and $A'$ differ for the first time on the path, shows that also the sign vectors differ. 
\end{proof}

Taking the consideration of Example~\ref{ex:single+matching} into account, we obtain the cardinality of $\omm$.

\begin{corollary}
  For a linkage matching field, the set $\omm$ contains $2^{d(n-d)+1}$ different sign maps. 
\end{corollary}

  Two chirotopes for uniform matroids are isomorphic, if they differ by a reorientation or a relabeling; see~\cite[\S 3]{BLSWZ:1993} for more details on these notions. 
  In particular, for a fixed chirotope, there are at most $2^n \cdot n!$ isomorphic chirotopes.

\begin{corollary}
  If $(M_{\sigma})$ is a polyhedral matching field, then $\omm\left[(M_{\sigma})\right]$ contains at least $2^{d(n-d)-n+1} / n!$ non-isomorphic chirotopes.
  For $2\log_2(n) < d < n / \log_2(n)$ with $n > 8$, this shows that there are non-isomorphic chirotopes. 
\end{corollary}
\begin{proof}
  The first part of the claim just follows from the estimate on the number of isomorphic chirotopes. 
  
  For the second part, we use the estimate $n! < 2^{n\log_2 n}$ and get
  \[
  2^{d(n-d)-n+1} / n! > 2^{d(n-d)-n+1-n\log_2(n)} \enspace .
  \]
  Thus, we derive a bound for the exponent, using the specified range of $d$: 
  \begin{align*}
    d(n-d)-n+1-n\log_2(n) &>& 2 \log_2(n) (n - \frac{n}{\log_2(n)}) - n + 1 - n\log_2(n) \\
    &=& 2 n\log_2(n) - 2 n - n + 1 - n\log_2(n) \\
    &=& n \log_2(n) - 3n + 1 \geq 1 .
  \end{align*}
\end{proof}

\begin{question}
  For a polyhedral matching field, how many non-isomorphic chirotopes does $\omm$ contain?
\end{question}

We can also vary the matching field and observe how the set of sign maps differs.
If two matching fields do not have the same support, one can choose a sign matrix $A \in \{+,-\}^{d \times n}$ such that $\rchi((M_{\sigma}),A)$ and $\rchi((M'_{\sigma}),A)$ are distinct; that is simply by modifying the entries which are in the difference of the supports.
Though, for two (even coherent!) matching fields with different support still the set of chirotopes $\omm$ can agree as the next example shows.

\begin{example}
  The two coherent $(3,4)$-matching fields depicted in Figure~\ref{fig:two+linkage+covectors} by their linkage pd-graphs give rise to the same set $\omm$.
  They agree on all matchings except for the one on $\{2,3,4\}$; however, they differ there by a permutation with positive sign. 
\end{example}

\begin{figure}
  \centering
  \begin{tikzpicture}[scale=0.6]

\bigraphthreefourcoord{0}{0}{1.2}{1.3}{2.5};
\draw[EdgeStyle] (v1) to (w1);
\draw[EdgeStyle] (w1) to (v2);
\draw[EdgeStyle] (v2) to (w2);
\draw[EdgeStyle] (w2) to (v3);
\draw[EdgeStyle] (v3) to (w3);
\draw[EdgeStyle] (w3) to (v4);
\bigraphthreefournodes

\bigraphthreefourcoord{6}{0}{1.2}{1.3}{2.5};
\draw[EdgeStyle] (v1) to (w1);
\draw[EdgeStyle] (w1) to (v4);
\draw[EdgeStyle] (v2) to (w2);
\draw[EdgeStyle] (w2) to (v3);
\draw[EdgeStyle] (v3) to (w3);
\draw[EdgeStyle] (w3) to (v4);
\bigraphthreefournodes

\end{tikzpicture}
    \caption{Two coherent matching fields giving rise to the same set $\omm$. }
  \label{fig:two+linkage+covectors}
\end{figure}
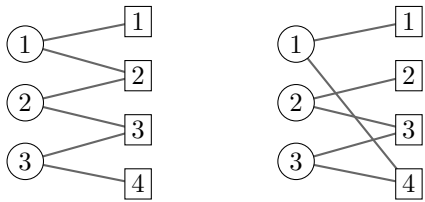

\begin{conjecture}
  A matching field is polyhedral if and only if the sign map induced by each sign matrix is a chirotope. 
\end{conjecture}

\begin{remark}
  The study of coherent matching fields is closely related to the structure of the tropical maximal minors of rectangular matrices; see~\cite{FinkRincon:2015}.  
  In this case, considering subpolytopes of $\ssimplex_{d-1} \times \ssimplex_{n-1}$ can be thought as setting some entries of the height matrix to tropical zero. 
  Over the $\min$-plus semiring, this means that non-bases of the transversal matroid are those submatrices whose tropical determinant is $\infty$.
  It is natural to consider tropical singularity,  another tropical analogue of zero determinant.
  A square matrix is {\em tropically singular} if the minimum in its tropical determinant is achieved at least twice.
  The polyhedral interpretation is that we work with general subdivisions instead of triangulations, and only consider simplices in the subdivision as bases.
  However, the following example shows that this does not always produce a matroid.
  A conceptual explanation is that a system of tropical non-singularity conditions can not always be lifted back to non-singularity over a (valued) field.
\end{remark}

\begin{example}
Consider the matrix

\[
\begin{pmatrix}
0 & 0 & 0 & 0 & 0\\
0 & 1 & 1 & 2 & 2\\
0 & 0 & 1 & 1 & 2
\end{pmatrix} \;.
\]

The (column sets of) $\min$-tropically non-singular submatrices are
$$123, 124, 125, 145, 234, 235, 345,$$
which do not form a collection of bases of any matroid.
This can be seen from the pair of bases $123$ and $145$ by removing $2$ from the first set. 
\end{example}

\subsection{Ringel's non-realizable $(3,9)$ uniform oriented matroid} \label{sec:ringel}

There is a unique (up to isomorphism) non-realizable uniform oriented
matroid $\Rin$ of rank $3$ on $9$ elements, which goes back to
Ringel~\cite[\S 8.3]{BLSWZ:1993}. We show that $\Rin$ can be realized
as the oriented matroid $\widetilde{\orientedmatroid}$ from the pointed
polyhedral matching field $(\widetilde{M}_{\sigma})$ associated to
a triangulation of $\ssimplex_{5}\times\ssimplex_{2}$ ({\em a posteriori}
non-regular) together with an appropriate sign matrix. Any rank 3
oriented matroid can be represented geometrically as a \emph{pseudoline
arrangement}; that is, a collection of curves in the plane which pairwise
intersect exactly once. A pseudoline arrangement for $\Rin$ is depicted
in Figure~\ref{fig:Ringel_Gruenbaum} below; see also~\cite[p.42]{Gruenbaum:1972}.

\begin{figure}[H]
\centering{}\includegraphics[scale=0.3]{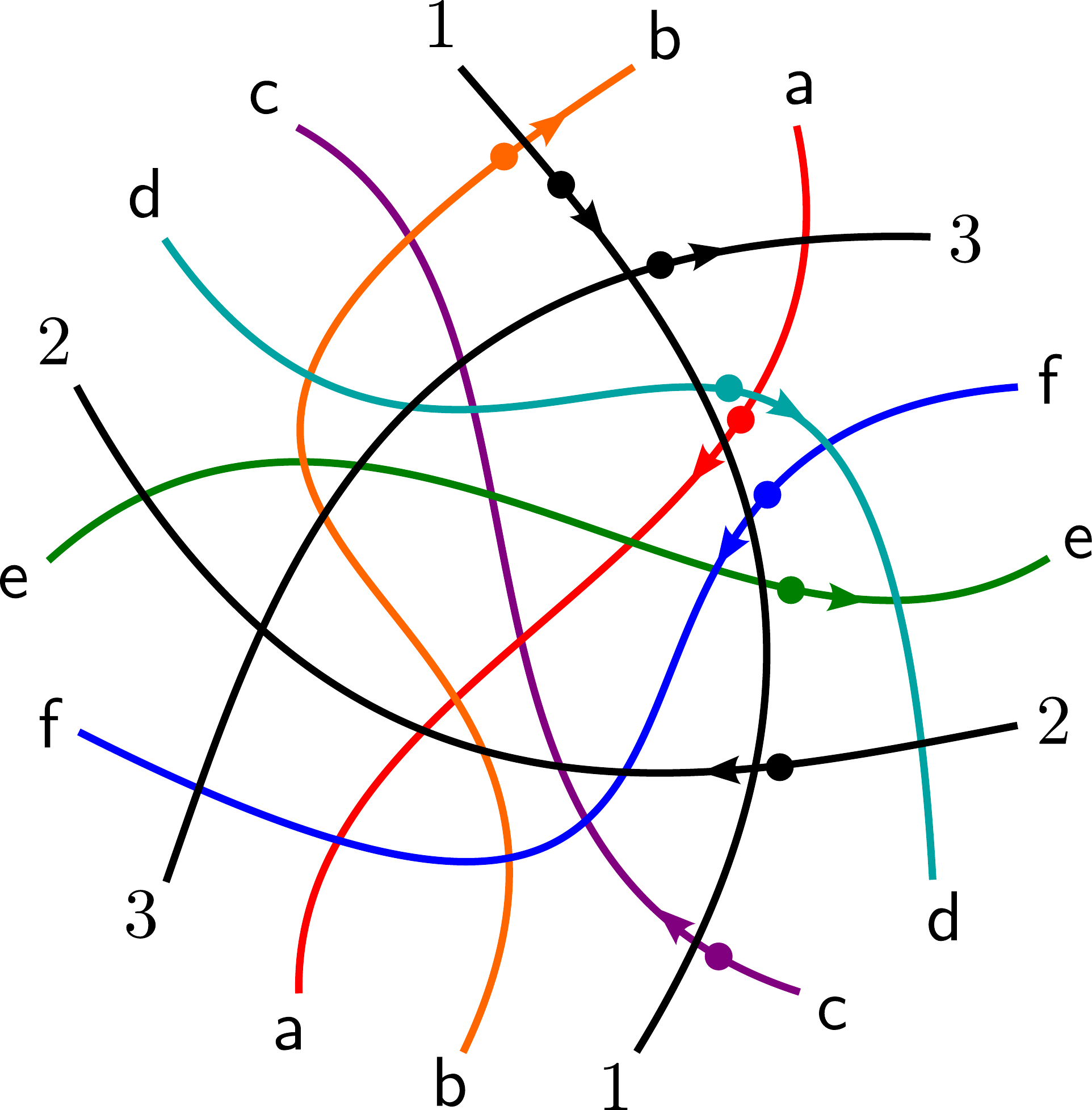}\caption{\label{fig:Ringel_Gruenbaum}A pseudoline arrangement representing
$\protect\Rin$.}
\end{figure}

The triangulation giving rise to $\Rin$ is shown below in Figure~\ref{fig:Ringel_Cayley_trick}
as a fine mixed subdivision of $6\ssimplex_{2}$ (cf. the Cayley trick
in Section~\ref{sec:triangulation}). It has appeared already in
the work of Ardila and Ceballos, where it was provided by Hwanchul
Yoo as a counterexample to a conjecture of theirs about acyclic systems
of permutations; see~\cite[Conjecture 5.7]{ArdilaCeballos:2013}
for further details.

\begin{figure}[H]
\centering{}\includegraphics[scale=0.7]{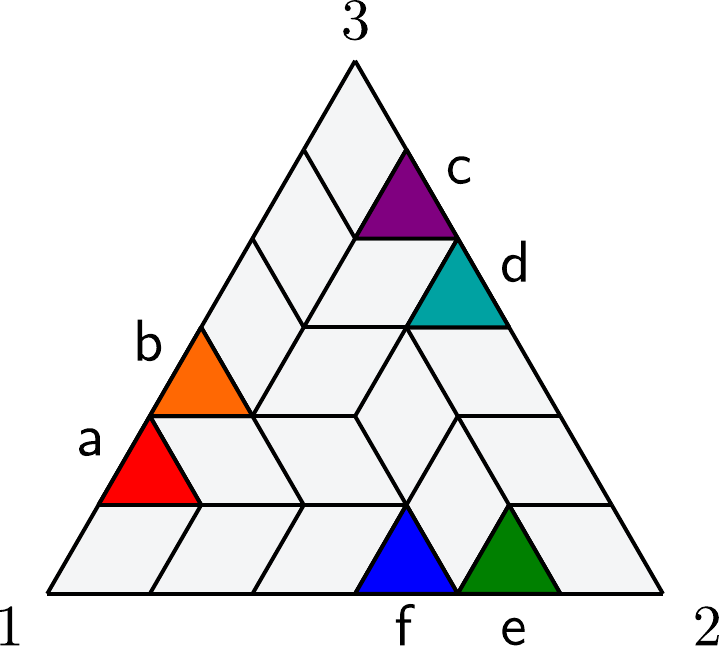}\caption{\label{fig:Ringel_Cayley_trick}The Cayley trick representation of
the triangulation.}
\end{figure}
We also fix the sign matrix

\[
A=\begin{pmatrix}\phantom{-}1 & -1 & -1 & \phantom{-}1 & \phantom{-}1 & \phantom{-}1\\
-1 & \phantom{-}1 & \phantom{-}1 & \phantom{-}1 & \phantom{-}1 & -1\\
\phantom{-}1 & \phantom{-}1 & \phantom{-}1 & -1 & -1 & \phantom{-}1
\end{pmatrix}
\]
and we set $\widetilde{A}=\left(I_{3\times3}\mid A\right)$. Identifying
the ground set $\ground=\left\{ \mathsf{a},\mathsf{b},\mathsf{c},\mathsf{d},\mathsf{e},\mathsf{f}\right\} $
with the six triangles constituting the Minkowski summands of $6\ssimplex_{2}$,
we may alternatively represent this data as shown in Figure~\ref{fig:Ringel_sign_matrix}, where the
solid vertices denote $+$ and the hollow vertices denote $-$:

\begin{figure}[H] 
\centering{}\includegraphics[scale=0.6]{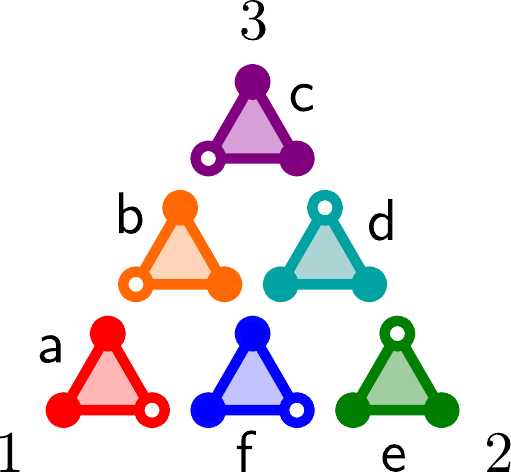}\caption{\label{fig:Ringel_sign_matrix}Another representation of the sign
matrix $A$}
\end{figure}
We depict in Figure~\ref{fig:Depicting_(S,F)} all pairs $(S,F)$
as in Proposition~\ref{prop:covectors+from+topes} in a way that
reflects the geometry of the given fine mixed subdivision of $6\ssimplex_{2}$.
How we do this is illustrated by example in Figure~\ref{fig:Depicting_(S,F)}(b)
below. Each forest $F$ corresponds to a cell in the fine mixed subdivision,
and we depict this cell by its Minkowski sum decomposition into faces
of $\ssimplex_{2}$. The sign vector $S$ then dictates the signs
on each vertex relative to the signs of Figure~\ref{fig:Ringel_sign_matrix}.
The support of $S$ is also shown in gray.

%% \begin{figure}[H]
%% \centering{}
%% \subcaptionbox{In this example, $F$ is defined to be the tree shown above including
%% the dotted edges, and $F'$ is obtained from $F$ by deleting the
%% dotted edges.}
%% {\qquad{}\enskip{}\includegraphics[scale=0.6]{signed_forest_graph}\qquad{}\enskip{}}
%% \quad{}\subcaptionbox{The pairs $(S,F)$ (top), and $(S',F')$ (bottom), where $S=\left(+,+,-\right)$
%% and $S'=\left(+,0,-\right)$.}
%% {\qquad{}\enskip{}\includegraphics[scale=0.6]{two_signed_forests}\qquad{}\enskip{}}
%% \quad{}\subcaptionbox{The covectors $\varphi_{\widetilde{A}}(S,F)$ and $\varphi_{\widetilde{A}}(S',F')$.
%% The sign vectors $S$ and $S'$ are shown on the left.}
%% {\qquad{}\includegraphics[scale=0.6]{two_signed_forests_covectors}\qquad{}}\caption{\label{fig:Depicting_(S,F)}Depicting the pairs $(S,F)$ as in Proposition~\ref{prop:covectors+from+topes}.}
%% \end{figure}

\begin{figure}[H]
  \centering{}
  \begin{subfigure}[h]{0.28\textwidth}
      \centering{}
    \includegraphics[scale=0.6]{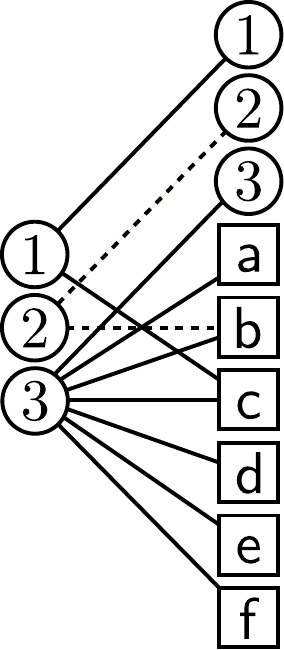}
      \caption{$F$ is defined to be the tree shown above including the dotted edges, and $F'$ is obtained from $F$ by deleting the dotted edges.}
  \end{subfigure}
  \quad
  \begin{subfigure}[h]{0.28\textwidth}
      \centering{}
  \includegraphics[scale=0.6]{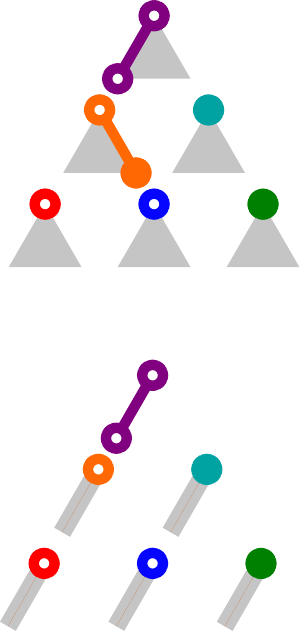}
  \caption{The pairs $(S,F)$ (top), and $(S',F')$ (bottom), where $S=\left(+,+,-\right)$ and $S'=\left(+,0,-\right)$.}
  \end{subfigure}
  \quad
  \begin{subfigure}[h]{0.28\textwidth}
      \centering{}
    \includegraphics[scale=0.6]{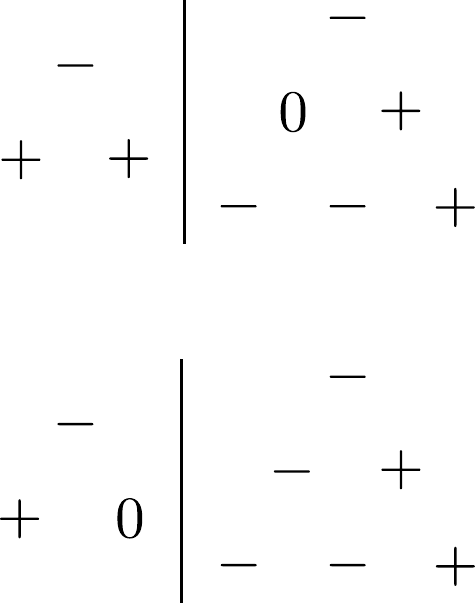}
    \caption{The two covectors $\varphi_{\widetilde{A}}(S,F)$ and $\varphi_{\widetilde{A}}(S',F')$. The sign vectors $S$ and $S'$ are shown on the left.}
  \end{subfigure}
  \caption{\label{fig:Depicting_(S,F)}Depicting the pairs $(S,F)$ as in Proposition~\ref{prop:covectors+from+topes}.}
\end{figure}

\begin{figure}[H]
  \begin{subfigure}[h]{\textwidth}
\centering
\includegraphics[height=0.9\textheight]{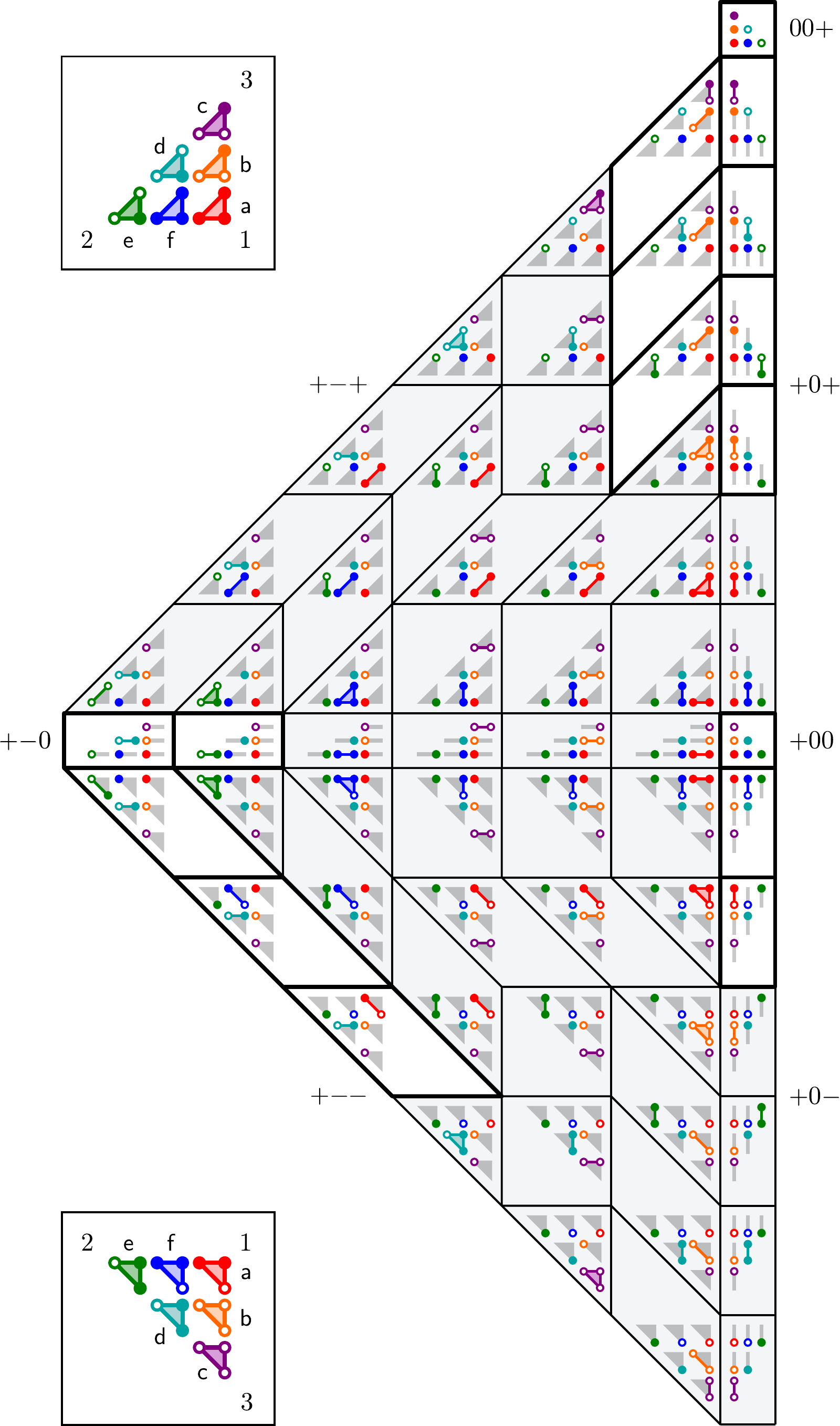}
\caption{Left part}
  \end{subfigure}
\end{figure}

\begin{figure}[H] \ContinuedFloat
  \begin{subfigure}[h]{\textwidth}
    \includegraphics[height=0.9\textheight]{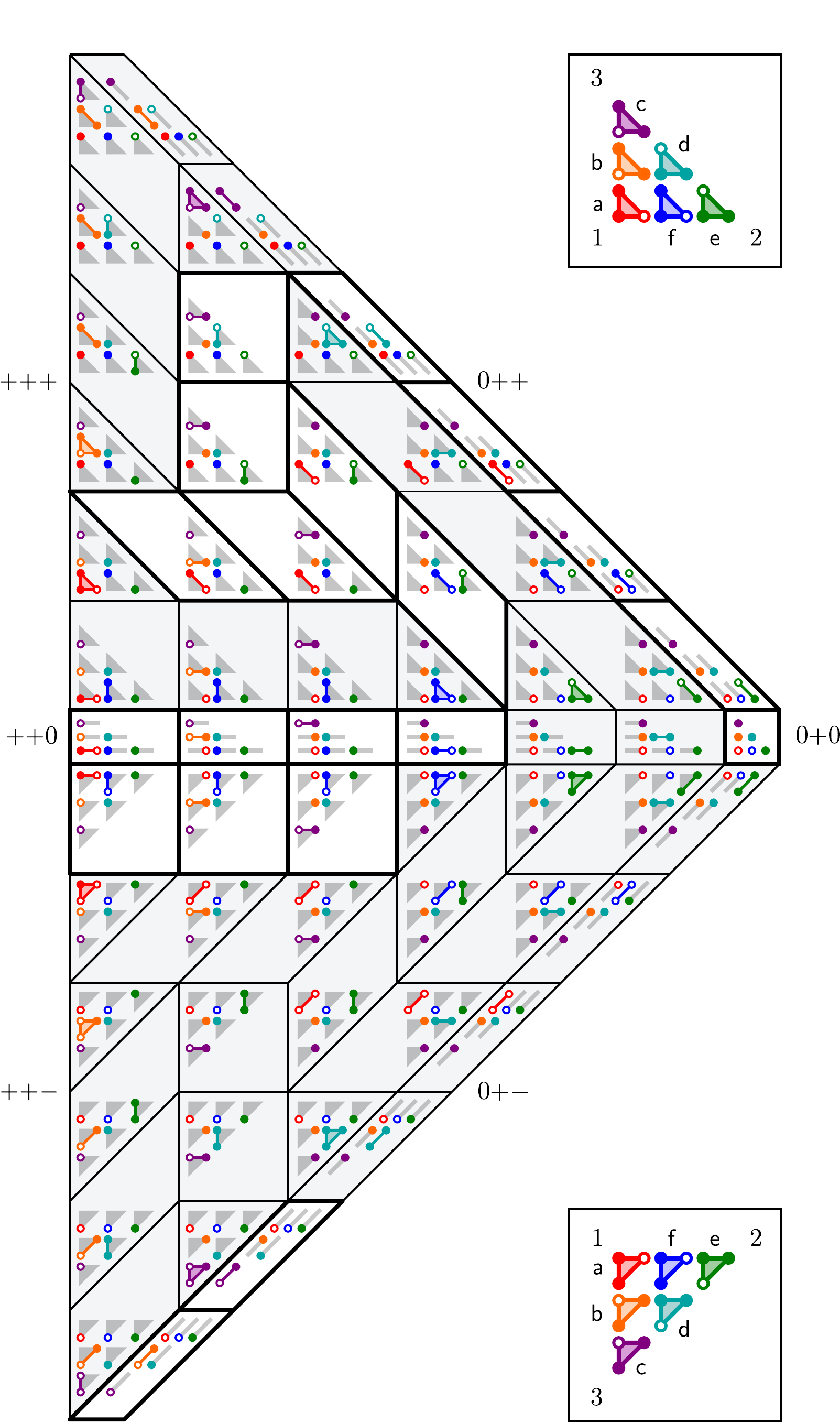}
    \caption{Right part}
  \end{subfigure}
  \caption{\label{fig:Ringel_singed_trees}Pairs $(S,F)$ grouped according to the sign vector $S$.}
\end{figure}

%\todo[inline]{GL: I still would prefer to have this Figure on one page; if not, then at least the figures should be subfigures.  }

It follows from Proposition~\ref{prop:cocircuit} that a pair $(S,F)$
is mapped to a cocircuit of $\widetilde{\orientedmatroid}$ under
$\psi_{\widetilde{A}}$ if and only if there are exactly $\left|\supp(S)\right|-1$
Minkowski summands of $F$ that are segments with oppositely signed
vertices. In Figure~\ref{fig:Ringel_singed_trees}, such pairs are
shown in bold. Now consider the following pseudoline arrangement with
$9$ pseudolines, oriented so that the region containing the large
black dot is contained in the all $+$ tope:

\begin{figure}[H]
\centering{}\includegraphics[scale=0.35]{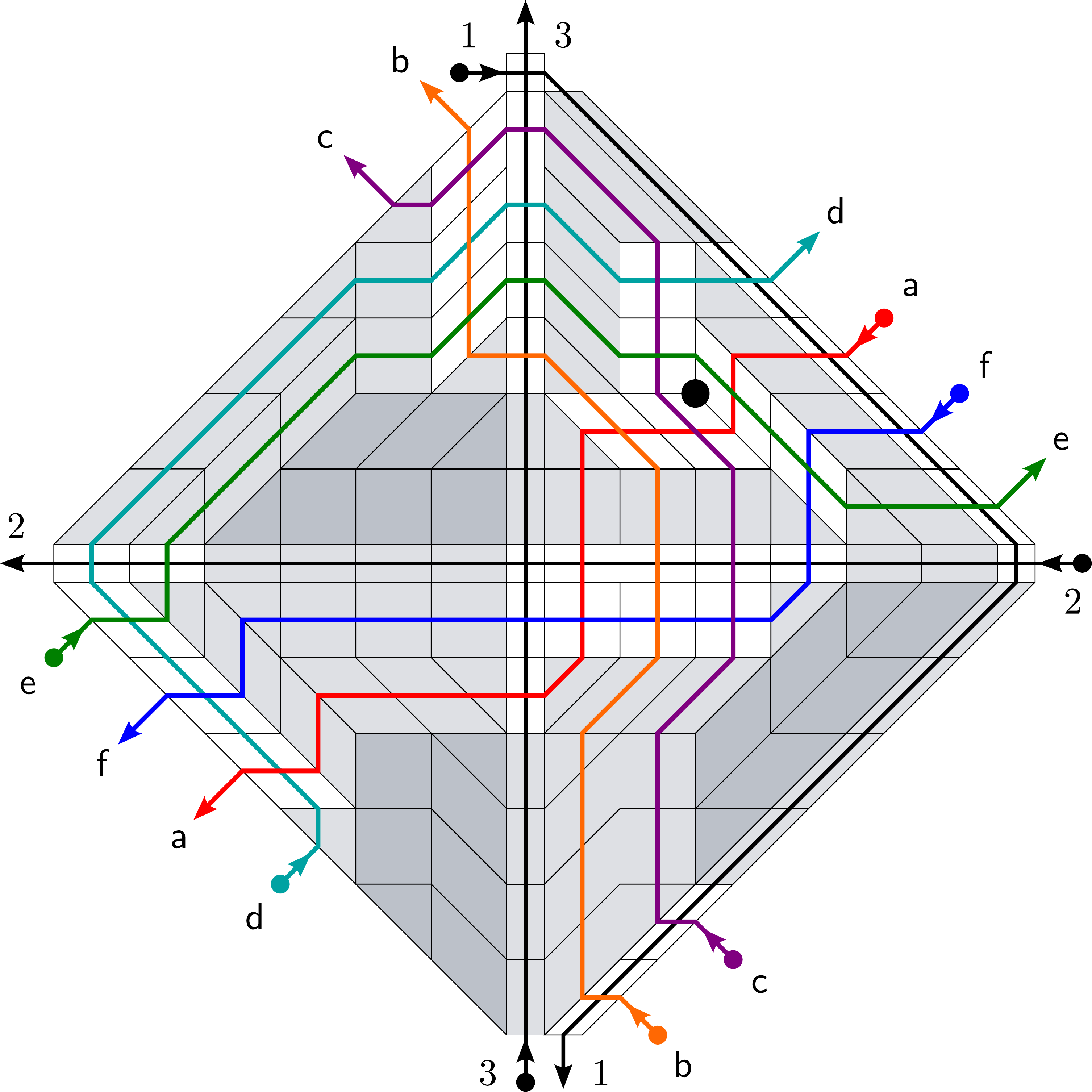}\caption{\label{fig:Ringel_pseudolines_patchworking}The pseudoline arrangement
derived from Figure~\ref{fig:Ringel_singed_trees}.}
\end{figure}
It can be verified that the $\tbinom{9}{2}$ cocircuits shown in Figure~\ref{fig:Ringel_pseudolines_patchworking}
above are precisely the same as those arising from Figure~\ref{fig:Ringel_singed_trees}.
Moreover, in both figures, the remaining $\tbinom{9}{2}$ cocircuits
``on the other side'' are obtained by negating the ones shown. As
oriented matroids are determined by their cocircuits, we see that
this pseudoline arrangement must therefore be a representation of
$\widetilde{\orientedmatroid}$.

The last step now is to verify that this pseudoline arrangement also
represents Ringel's oriented matroid $\Rin$. One may check these
pseudoline arrangements are homeomorphic by considering their \emph{pseudoline
sequences}. Each pseudoline is intersected by all other pseudolines,
giving rise to a (circular) sequence of elements of the oriented matroid
for each pseudoline. In~\cite[Thm.~5.2]{BokowskiMockStreinu:2001},
it was shown that the set of these sequences determines the oriented
matroid.

Using the starting dots and directions indicated, it can be verified
that the sequences for Figures~\ref{fig:Ringel_Gruenbaum}~and~\ref{fig:Ringel_pseudolines_patchworking}
agree. For example, the pseudoline sequence associated to $\mathsf{a}$
in both figures is $1\mathsf{ecb}2\mathsf{f}3\mathsf{d}$.

\section{Extension to Matroids over Hyperfields} \label{sec:hyperfield}

\subsection{Introduction to Matroids over Hyperfields}

\begin{definition}
A {\em hyperfield} $(\mathbb{H},\boxplus,\otimes,0,1)$ consists of a set $\mathbb{H}$ with distinguished elements $0\neq 1$, together with a possibly multi-valued hyperoperation $\boxplus:\mathbb{H}\times\mathbb{H}\rightarrow 2^{\mathbb H}\setminus\{\emptyset\}$ and an operation $\otimes:\mathbb{H}\times\mathbb{H}\rightarrow \mathbb{H}$, such that:
\begin{itemize}
\item $x\boxplus y = y\boxplus x,\forall x,y\in\mathbb{H}$.
\item $\bigcup_{a\in x\boxplus y} a\boxplus z=\bigcup_{b\in y\boxplus z} x\boxplus b=: x\boxplus y\boxplus z, \forall x,y,z\in\mathbb{H}$.
\item $0\boxplus x=\{x\},\forall x\in\mathbb{H}$.
\item For any $x\in\mathbb{H}$, there exists a unique $-x\in\mathbb{H}$ such that $0\in x\boxplus(-x)$.
\item $x\in y\boxplus z$ if and only if $z\in x\boxplus (-y)$.
\item $(\mathbb{H}\setminus\{0\},\otimes, 1)$ is an abelian group, and $0\otimes x=x\otimes 0=0, \forall x\in\mathbb{H}$.
\item $a\otimes(x\boxplus y)=(a\otimes x)\boxplus (a\otimes y), \forall a,x,y\in\mathbb{H}$.
\end{itemize}

Given two hyperfields $(\mathbb{H},\boxplus,\otimes,0,1), (\mathbb{H}',\boxplus',\otimes',0',1')$, a map $\varphi:\mathbb{H}\rightarrow\mathbb{H}'$ is a {\em hyperfield morphism} if $\varphi(0)=0',\varphi(1)=1', \varphi(x\otimes y)=\varphi(x)\otimes'\varphi(y), \forall x,y\in\mathbb{H}$, and $\varphi(x\boxplus y)\subset \varphi(x)\boxplus' \varphi(y),\forall x,y\in\mathbb{H}$.
\end{definition}

\begin{definition}
Let $\mathbb{H}=(\mathbb{H},\boxplus,\otimes,0,1)$ be a hyperfield. 
A {\em strong matroid over $\mathbb{H}$} (on $\ground$ and of rank $d$) is given by a non-zero, alternating function $\rchi:\ground^d\rightarrow\mathbb{H}$ such that for any $x_1,\ldots,x_{d-1},y_1,\ldots, y_{d+1}\in \ground$,
$$
0\in\boxplus_{k=1}^{d+1} (-1)^k \rchi(x_1,\ldots,x_{d-1},y_k)\rchi(y_1,\ldots,\widehat{y_k},\ldots,y_{d+1}).
$$

A non-zero alternating function $\rchi:\ground^d\rightarrow\mathbb{H}$ specifies a {\em weak matroid over $\mathbb{H}$} if $\underline{\rchi}$ is a matroid and that for any $x_1,x_2,y_1,y_2 \in \ground,X:=\{x_3,\ldots,x_d\}\subseteq \ground$,
$$
0\in \rchi(x_1,x_2,X)\rchi(y_1,y_2,X)\boxplus\rchi(x_1,y_1,X)\rchi(y_2,x_2,X)\boxplus\rchi(x_1,y_2,X)\rchi(x_2,y_1,X).
$$

In both cases, two functions differing by a non-zero scalar multiple represent the same matroid.
\end{definition}

Note the similarity between the definitions of a strong and a weak matroid and the two \emph{equivalent} characterizations of an oriented matroid given in Section~\ref{sec:OM}.

We state a proposition that the theory of matroids over hyperfields is functorial.

\begin{proposition}[{\cite[Lemma~3.40]{BakerBowler:2019}}] \label{prop:functorial}
Let $\varphi:\mathbb{H}\rightarrow\mathbb{H}'$ be a hyperfield morphism, and $\rchi$ be a (strong, respectively weak) matroid over $\mathbb{H}$.
Then $\varphi_*\rchi:=\varphi\circ\rchi$ is a (strong, respectively weak) matroid over $\mathbb{H}'$.
\end{proposition}

We collect a few essential examples here and refer the reader to \cite{BakerBowler:2019} for more, in particular their Examples~3.37 and~3.38, which show that in general, the notions of strong and weak matroids do not coincide.

\begin{example}
We omit the arithmetic of hyperfields in the list that can be directly deduced from the axioms.
\begin{itemize}
\item Every field is a hyperfield. 
A (strong or weak) matroid over a field is a linear subspace over the field (represented by its Pl\"{u}cker coordinates).
\item The {\em Krasner hyperfield} $\mathbb{K}=\{0,1\}$ with $1\boxplus 1=\{0,1\}$. 
A (strong or weak) matroid over $\mathbb{K}$ is a matroid in the usual sense.
\item The {\em sign hyperfield} $\mathbb{S}=\{0,1,-1\}$ with $1\boxplus 1=\{1\}, -1\boxplus-1=\{-1\}, 1\boxplus -1=\{0,1,-1\}$. 
A (strong or weak) matroid over $\mathbb{S}$ is an oriented matroid.
\item  The {\em $\min$-tropical hyperfield} $\mathbb{T}=\mathbb{R}\cup\{\infty\}$, where $0\in\mathbb{R}$ is the multiplicative identity and $\infty$ is the additive identity, $a\boxplus b=\min\{a,b\}$ if $a\neq b$, $a\boxplus a=[a,\infty]$, and $a\otimes b=a+b$. 
A (strong or weak) matroid over $\mathbb{T}$ is a {\em valuated matroid} (also know as a tropical linear space).
\item The {\em phase hyperfield} $\mathbb{P}=\{z\in\mathbb{C}:|z|=1\}\cup\{0\}$, where $w\boxplus z$ consists of the open minor arc between $w,z$ if $w\neq -z$ are non-zero, $w\boxplus(-w)=\{w,-w,0\}$, and $w\otimes z=wz$. 
Matroids over $\mathbb{P}$, known as {\em phase matroids}, were first considered by Anderson and Delucchi \cite{AndersonDelucchi:2012}. 
It is an important example because the notions of strong and weak matroids are different over $\mathbb{P}$.
\end{itemize}
\end{example}

\begin{remark}
The more general notion of {\em matroids over tracts} was considered in \cite{BakerBowler:2019}, and there is a straightforward generalization of our work at that generality. 
However, the exposition of such theory is more complicated and goes beyond the scope of this work, so we omit the details here.
\end{remark}

\subsection{Hyperfields with the Inflation Property} \label{sec:IP}

The following definition was probably first considered by Massouros \cite{Massouros:1991} under the name of {\em monogene hyperfields}, but we follow the terminology of Anderson \cite{Anderson:2019}.

\begin{definition}
A hyperfield $\mathbb{H}$ has the {\em inflation property} (IP) if $1\boxplus (-1)=\mathbb{H}$.
\end{definition}

The following proposition is \cite[Proposition 6.10]{Anderson:2019}.

\begin{proposition}\label{prop:IP}
The statements are equivalent for a hyperfield $\mathbb{H}$:
\begin{enumerate}
\item $\mathbb{H}$ has the IP.
\item $a\boxplus (-a)=\mathbb{H}$ for any $a\neq 0$.
\item $a\in a\boxplus b$ for any $a\neq 0$.
\item Suppose $0\in\boxplus_{i=1}^{k} a_i$ for some $a_1,\ldots,a_k$ that are not all zero. Then $0\in (\boxplus_{i=1}^{k} a_i)\boxplus a_{k+1}$ for any $a_{k+1}\in\mathbb{H}$.
\end{enumerate}
\end{proposition}

We show that Theorem~\ref{thm:main} can be extended to (weak) matroids over hyperfields with the IP, this includes Theorem~\ref{thm:main} as a special case as $\mathbb{S}$ has the IP.

\begin{theorem}\label{thm:main_HF}
Let $(M_\sigma)$ be a polyhedral matching field. Let $\mathbb{H}$ be a hyperfield with the IP and $A$ be an $\mathbb{H}$-matrix with no zero entries. 
Then $\rchi:\ground^d\rightarrow\mathbb{H}$ given by $\sigma\mapsto \sign(M_\sigma)\otimes\bigotimes_{e\in M_\sigma}A_e$ is a weak matroid over $\mathbb{H}$.
\end{theorem}

\begin{proof}
We show that the proof in Section~\ref{sec:main} can be adapted to this setting. 
First of all, Theorem~\ref{thm:matroid+subdivision+from+triangulation} is a pure polyhedral geometric statement which requires no modification.

Following the argument of Lemma~\ref{lem:TM_real}, each maximal minor of $A_T$ is either 0 or consists of one non-zero term, so they are all single valued and altogether realize $\rchi_T$.
Now we claim that the maximal minors of an $\mathbb{H}$-matrix whose support is a spanning tree induce a weak matroid over $\mathbb{H}$. 
Treating the non-zero entries as algebraically independent indeterminates and denoting such matrix as $\tilde{A}$, we still have the $3$-term GP relation:
$\det(\tilde{A}_{|x_1,x_2,X})\det(\tilde{A}_{|y_1,y_2,X})+\det(\tilde{A}_{|x_1,y_1,X})\det(\tilde{A}_{|y_2,x_2,X})+\det(\tilde{A}_{|x_1,y_2,X})\det(\tilde{A}_{|x_2,y_1,X})=0$. 
Since each determinant is either 0 or a monomial, either all three terms are zeros, or one term is zero while the other two are the same monomial with coefficient $1$ and $-1$, respectively. 
Replacing the indeterminates by values from $\mathbb{H}$ will still preserve the $3$-term GP relation over $\mathbb{H}$.

Finally, the only non-polyhedral geometric argument in the proof of Theorem~\ref{thm:local_global} is to show that the restriction of a violation of the 3-term GP relation (in the ambient matroid polytope) to a subpolytope is still a violation.
The only non-trivial restriction is from an octahedron face to a pyramid/square cell.
In the language of hyperfields this means that if $0\not\in a\boxplus b\boxplus c$ with $a,b,c\neq 0$, then $0$ is also not in the sum if we set one of the terms to zero.
But this is Property~(4) of Proposition~\ref{prop:IP} concerning hyperfields with the IP.
\end{proof}

\begin{example}
We note that the statement of Theorem~\ref{thm:main_HF} actually characterizes hyperfields with the IP. 
Suppose $\mathbb{H}$ does not have the IP, pick $a\in\mathbb{H}$ such that $-a\not\in 1\boxplus (-1)$, i.e., $0\not\in 1\boxplus (-1)\boxplus a$. 
Take $n=4, d=2$ and the diagonal matching field of the corresponding size depicted in Figure~\ref {fig:four+two+matching+field}. 
Also take the $\mathbb{H}$-matrix to be

$$
\begin{pmatrix}
1 & 1 & 1 & 1\\
1 & a & 1 & 1
\end{pmatrix}.
$$

The induced $\rchi$ is not a (weak) matroid over $\mathbb{H}$.
\end{example}

We list several examples and constructions of hyperfields that have the IP, besides $\mathbb{K}$ and $\mathbb{S}$.

\begin{itemize}
\item The {\em tropical phase hyperfield} $\Phi=\{z\in\mathbb{C}:|z|=1\}\cup\{0\}$, where $w\boxplus z$ consists of the closed minor arc between $w,z$ if they are both non-zero and not antipodal of each other, $w\boxplus(-w)=\Phi$ if $w\neq 0$, and $w\otimes z=wz$.
\item Let $(G,\otimes,1)$ be an arbitrary abelian group. Introduce a new element $0$ and define the hyperoperation $\boxplus$ as $a\boxplus b=\{a,b\}$ if $a,b\in G$ and $a\neq b$, $a\boxplus a=G\cup\{0\}$ if $a\in G$, and $a\boxplus 0=0\boxplus a=\{a\}$.
Then $(G\cup\{0\},\boxplus, \otimes, 0, 1)$ is a hyperfield with the IP. 
Such hyperfields are considered in \cite{Massouros:1991}.
\item Again start with an arbitrary abelian group $(G,\otimes,1)$ together with a new element $0$.
Define $\boxplus'$ as $a\boxplus' b=G$ if $a,b\in G$ and $a\neq b$, $a\boxplus' a=G\cup \{0\}$, and $a\boxplus' 0=0\boxplus' a=\{a\}$.
Then $(G\cup\{0\},\boxplus', \otimes, 0, 1)$ is a hyperfield with the IP. 
Such hyperfields are called {\em weak hyperfields} in \cite{BakerBowler:2019}.
\item Let $\mathbb{H}=(\mathbb{H},\boxplus,\otimes,0,1)$ be a hyperfield.
Define a new hyperfield $\widetilde{\mathbb H}$ with the same ground set and multiplicative structure as $\mathbb{H}$, but with the new hyperoperation $\widetilde{\boxplus}$ where $a\widetilde{\boxplus} b=(a\boxplus b)\cup\{a,b\}$ for $a,b\neq 0$ and $a\neq -b$, and $a\widetilde{\boxplus} (-a)=\mathbb{H}$ for non-zero $a$. 
See \cite{Massouros:1991} for a proof that it is indeed a hyperfield.
\end{itemize}

We further elaborate our last example above. 
We suggest the name {\em canonical inflation} for the construction of $\widetilde{\mathbb H}$ from $\mathbb{H}$: in view of Proposition~\ref{prop:IP}, this is the ``minimum'' change needed to make a hyperfield to become one that has the IP. 
More rigorously, the (set-theoretic) identity map $\iota:\mathbb{H}\rightarrow\widetilde{\mathbb H}$ is a hyperfield morphism and we have the following universal property:

\begin{proposition} \label{prop:IP_universal}
Let $\varphi:\mathbb{H}\rightarrow\mathbb{H}'$ be a hyperfield morphism with $\mathbb{H}'$ having the IP.
Then we have the factorization $\mathbb{H}\xrightarrow{\iota}\widetilde{\mathbb H}\xrightarrow{\varphi}\mathbb{H}'$ of hyperfield morphisms.
\end{proposition}

\begin{proof}
The only non-trivial part is that $\widetilde{\mathbb H}\xrightarrow{\varphi}\mathbb{H}'$ preserves addition of two non-zero values.
Denote by $\boxplus, \widetilde{\boxplus}, \boxplus'$ the addition hyperoperators of $\mathbb{H}, \widetilde{\mathbb H}, \mathbb{H}'$, respectively.
Suppose $a\neq -b$ are non-zero. Then $\varphi(a)\boxplus'\varphi(b)$ contains $\varphi(a\boxplus b)$ (as $\mathbb {H}\xrightarrow{\varphi}\mathbb{H}'$ is a hyperfield morphism) as well as $\{\varphi(a), \varphi(b)\}$ (as $\mathbb{H}'$ has the IP), so $\varphi(a)\boxplus'\varphi(b)$ contains $\varphi(a\widetilde{\boxplus} b)$.
For $a\neq 0$, we have $\varphi(a\widetilde{\boxplus} (-a))=\varphi(\mathbb{H})\subset\mathbb{H}'=\varphi(a)\boxplus' \varphi(-a)$.
\end{proof}

\begin{corollary} \label{coro:main_HF}
Let $(M_\sigma)$ be a polyhedral matching field. Let $\mathbb{H}$ be an arbitrary hyperfield and $A$ be an $\mathbb{H}$-matrix with no zero entries. 
Then $\rchi:\ground^d\rightarrow\mathbb{H}$ given by $\sigma\mapsto \sign(M_\sigma)\otimes\bigotimes_{e\in M_\sigma}A_e$ is a weak matroid over the canonical inflation $\widetilde{\mathbb H}$ of $\mathbb{H}$.
In general, given a hyperfield morphism  $\varphi:\mathbb{H}\rightarrow\mathbb{H}'$ with $\mathbb{H}'$ having the IP, $\varphi_*\rchi$ is a weak matroid over $\mathbb{H}'$.
\end{corollary}

\begin{proof}
Note that $\rchi$ does not involve the additive structure of $\mathbb{H}$ (respectively  $\widetilde{\mathbb H}$), so we can simply interpret $A$ as a $\widetilde{\mathbb H}$-matrix and apply Theorem~\ref{thm:main_HF} to $\rchi$.
The general statement follows from Proposition~\ref{prop:IP_universal} and Proposition~\ref{prop:functorial}.
\end{proof}

\begin{example} \label{ex:Phi_matroids}
Let $ph:\mathbb{C}\rightarrow\mathbb{P}$ be given by $z\mapsto z/|z|$ if $z\neq 0$ and $0\mapsto 0$. 
Then we have the hyperfield morphisms $\mathbb{C}\xrightarrow{ph}\mathbb{P}\xrightarrow{\iota}\Phi$, where the second map is the canonical inflation. 
A matroid over $\mathbb{P}$ or $\Phi$ is {\em $\mathbb{C}$-realizable} if it is the pushforward of some matroid over $\mathbb{C}$. 
Consider the $(2,4)$-diagonal matching field together with the matrix

$$
\begin{pmatrix}
1 & 1 & 1 & 1\\
1 & i & 1 & 1
\end{pmatrix}.
$$

The induced function $\rchi:\binom{[4]}{2}\rightarrow\{z\in\mathbb{C}:|z|=1\}$  takes the (ordered) basis $12$ to $i$ and every other basis to $1$. 
It is easy to check that $\rchi$ is a (weak) matroid over $\Phi$ but not over $\mathbb{P}$, nor it is $\mathbb{C}$-realizable. 
However, with the natural Euclidean topology of $\mathbb{C},\mathbb{P},\Phi$ and their corresponding Grassmannians (viewed as subsets of $\mathbb{C}P^5$), $\rchi$ is a limit point of the Grassmannians ${\rm Gr}_{\mathbb P}(2,4)$ as well as the subset of $\mathbb{C}$-realizable matroids: 
we can approximate $\rchi$ by $ph_*\rchi_R$ as $R\rightarrow\infty$, where $\rchi_R$ is the matroid over $\mathbb{C}$ realized by the matrix

$$
\begin{pmatrix}
e^R & e^R & e^R & e^R\\
1 & ie^R & e^{2R} & e^{3R}
\end{pmatrix}.
$$

This suggests the potential application of matroids induced by matching fields in the general theory of hyperfields and their Grassmannians (we refer the reader to \cite{AndersonDavis:2019} for further discussion).

\end{example}

\section{Conclusion} \label{sec:conclusion}

%\todo[inline]{Need to decide do we defer some discussion to Part II or not. GL: I suggest to leave it as is. }

With a new point of view and machinery, we extend the connections between several prominent objects in matroid theory, discrete geometry, and tropical geometry beyond the ``realizable'' territory.
We also initiate the study of the interaction between matroid subdivisions and signs, or more generally, hyperfields.
As mentioned throughout in our paper, many of these ideas are interesting in its own right and could potentially be applied to other settings.
Recall that the original motivation of matching fields comes from combinatorial commutative algebra and the corresponding algebraic geometry, it is also interesting to see if something can be said in these areas.
We end with a few more open problems, focusing on aspects that were not fully discussed in the main content.

\smallskip

As already mentioned in the introduction, our work adds a new piece to the connection between the quest for a strongly polynomial algorithm for linear programming and a (weakly) polynomial algorithm for mean payoff games.

Many constructions of (sub)exponential instances for pivot rules for the simplex method are derived from parity or mean payoff games, respectively~\cite{Friedmann:2011}.
This was achieved by relating strategy iteration for these games with pivoting in the simplex method.
On the positive side, one can solve mean payoff games using their equivalence with tropical linear programming~\cite{AkianGaubertGuterman:2012} by a tropicalized version of the simplex method~\cite{ABGJ-Simplex:A}. 
Our work gives the manifestation of this correspondence on the level of oriented matroids and matching fields, as further discussed in Remark~\ref{rem:connection+tropical+simplex}. 

An indicator on the hardness of mean payoff games comes from the richness of the oriented matroids arising for coherent matching fields.
If only a subclass of uniform oriented matroids arises from coherent matching fields, this would exhibit a deep structural difference between linear programming and tropical linear programming. 
On the other hand, if all realizable uniform oriented matroids arise from coherent matching fields, then a combinatorial algorithm for tropical linear programming would already solve linear programming efficiently. 

\begin{question}
  Which oriented matroids are realizable from a coherent matching field or a polyhedral matching field?
\end{question}

This adds well to the problems posed in Section~\ref{sec:OM_MF} and the second part of the question goes even further in the direction to get tools for representing not-necessarily realizable oriented matroids.
This is already interesting for pseudoline arrangements which occur frequently in combinatorics.
In this case, the question reduces to the study of fine mixed subdivisions of dilated triangles $n \ssimplex_{2}$, which are substantially better understood than for arbitrary $d > 3$, see~\cite{ArdilaBilley:2007, Santos:2005}.

\smallskip

The space of all matrices which give rise to a prescribed coherent matching field is a full-dimensional cone of the normal fan of the Newton polytope of the product of all maximal minors of a matrix of indeterminates~\cite{SturmfelsZelevinsky:1993}. 
Hence, the space of matrices which induce the same oriented matroid by the construction of Theorem~\ref{thm:main+theorem+intro} is a union $U$ of cones.
On the other hand, the realization space of an oriented matroid is in general a very complicated object~\cite{Mnev:1988}.
One can consider such a realization space over Puiseux series and take its tropicalization $V$, which is also a polyhedral complex; see~\cite{Joswig:2020} for more on tropicalization. 

\begin{question}
  What is the relation of the two sets $U$ and $V$? 
\end{question}

An answer to the latter question might give a new approach for the understanding of realization spaces.
This is even interesting for the case $d = 3$, as already uniform oriented matroids of rank $3$ have arbitrarily complicated realization spaces. 

\smallskip

Our proof of Theorem~\ref{thm:local_global}, as well as many results in the literature on matroid subdivisions, relies crucially on the reduction to the 3-term GP relation, which is a local condition and rather easy to check.
This is only good enough to guarantee a weak matroid in Theorem~\ref{thm:main_HF}, and the conclusion poses the obvious question of whether we can say something stronger.

\begin{question}
Is the function $\rchi$ in Theorem~\ref{thm:main_HF} always a strong matroid over $\mathbb{H}$ as well?
\end{question}

If one wants to apply a similar polyhedral approach to this problem, it is likely that one has to analyze the global structure of matroid subdivisions.
On the other hand, we do not rule out the possibility that counterexamples exist over {\em imperfect} hyperfields (with the IP).
In any case, this problem might provide a polyhedral angle to understand the differences between strong and weak matroids.

\medskip

{\sc Acknowledgement.} The first author was funded by the Deutsche Forschungsgemeinschaft (DFG, German Research Foundation) under Germany's Excellence Strategy -- The Berlin Mathematics Research Center MATH+(EXC-2046/1, project ID: 390685689).
%He is grateful to Josephine Yu for helpful discussions on patchworking, in particular pointing out the reference \cite{Sturmfels:1994a}.
The second author was supported by the European Research Council (ERC) Starting Grant ScaleOpt, No.~757481.
He also thanks Xavier Allamigeon and Mateusz Skomra for inspiring discussions and Ben Smith for helpful comments.
The third author was supported by Croucher Fellowship for Postdoctoral Research and Netherlands Organisation for Scientific Research Vici grant 639.033.514 during his affiliation to Brown University and University of Bern, respectively.
The first and third author both thank Chris Eppolito for references and discussion on hyperfields.
The second and third author both thank the {\em Tropical Geometry, Amoebas, and Polyhedra} semester program at Institut Mittag-Leffler for the introduction that started the project.
All three authors thank Laura Anderson and Josephine Yu for their feedback on an earlier draft of this paper.

% -------------------------------- bibliography-----------------------------
\bibliographystyle{amsplain}
\bibliography{pmf}

\providecommand{\bysame}{\leavevmode\hbox to3em{\hrulefill}\thinspace}
\providecommand{\MR}{\relax\ifhmode\unskip\space\fi MR }
\providecommand{\MRhref}[2]{%
  \href{http://www.ams.org/mathscinet-getitem?mr=#1}{#2}
}
\providecommand{\href}[2]{#2}
\begin{thebibliography}{10}

\bibitem{AkianGaubertGuterman:2012}
Marianne Akian, St{\'e}phane Gaubert, and Alexander Guterman, \emph{Tropical
  polyhedra are equivalent to mean payoff games}, Internat. J. Algebra Comput.
  \textbf{22} (2012), no.~1, 1250001, 43.

\bibitem{ABGJ-Simplex:A}
Xavier Allamigeon, Pascal Benchimol, St{\'e}phane Gaubert, and Michael Joswig,
  \emph{Tropicalizing the simplex algorithm}, SIAM J. Discrete Math.
  \textbf{29} (2015), no.~2, 751--795.

\bibitem{AllamigeonBenchimolGaubertJoswig:2018}
Xavier {Allamigeon}, Pascal {Benchimol}, St\'ephane {Gaubert}, and Michael
  {Joswig}, \emph{{Log-barrier interior point methods are not strongly
  polynomial.}}, {SIAM J. Appl. Algebra Geom.} \textbf{2} (2018), no.~1,
  140--178.

\bibitem{Anderson:2001}
Laura Anderson, \emph{Representing weak maps of oriented matroids}, European
  Journal of Combinatorics \textbf{22} (2001), no.~5, 579 -- 586.

\bibitem{Anderson:2019}
Laura Anderson, \emph{Vectors of matroids over tracts}, J. Combin. Theory Ser.
  A \textbf{161} (2019), 236--270.

\bibitem{AndersonDavis:2019}
Laura Anderson and James~F. Davis, \emph{Hyperfield {G}rassmannians}, Adv.
  Math. \textbf{341} (2019), 336--366.

\bibitem{AndersonDelucchi:2012}
Laura Anderson and Emanuele Delucchi, \emph{Foundations for a theory of complex
  matroids}, Discrete Comput. Geom. \textbf{48} (2012), no.~4, 807--846.

\bibitem{ArdilaBilley:2007}
Federico Ardila and Sara Billey, \emph{Flag arrangements and triangulations of
  products of simplices}, Adv. Math. \textbf{214} (2007), no.~2, 495--524.

\bibitem{ArdilaCeballos:2013}
Federico Ardila and Cesar Ceballos, \emph{Acyclic systems of permutations and
  fine mixed subdivisions of simplices}, Discrete and Computational Geometry
  \textbf{49} (2013), no.~3, 485--510.

\bibitem{ArdilaDevelin:2009}
Federico Ardila and Mike Develin, \emph{Tropical hyperplane arrangements and
  oriented matroids}, Math. Z. \textbf{262} (2009), no.~4, 795--816.

\bibitem{AHrkaniLamSpradlin:2020}
Nima Arkani-Hamed, Thomas Lam, and Marcus Spradlin, \emph{{Positive
  configuration space}}, 2020, preprint
  \href{https://arxiv.org/abs/2003.03904}{arXiv:2003.03904}.

\bibitem{BakerBowler:2019}
Matthew Baker and Nathan Bowler, \emph{Matroids over partial hyperstructures},
  Adv. Math. \textbf{343} (2019), 821--863.

\bibitem{Benchimol:2014}
Pascal Benchimol, \emph{Tropical aspects of linear programming}, Theses,
  \'{E}cole {P}olytechnique, December 2014.

\bibitem{BernsteinZelevinsky:1993}
David Bernstein and Andrei Zelevinsky, \emph{Combinatorics of maximal minors},
  J. Algebraic Combin. \textbf{2} (1993), no.~2, 111--121.

\bibitem{BLSWZ:1993}
Anders Bj\"{o}rner, Michel Las~Vergnas, Bernd Sturmfels, Neil White, and
  G\"{u}nter~M. Ziegler, \emph{Oriented matroids}, Encyclopedia of Mathematics
  and its Applications, vol.~46, Cambridge University Press, Cambridge, 1993.

\bibitem{BokowskiMockStreinu:2001}
J\"urgen {Bokowski}, Susanne {Mock}, and Ileana {Streinu}, \emph{{On the
  Folkman-Lawrence topological representation theorem for oriented matroids of
  rank 3.}}, {Eur. J. Comb.} \textbf{22} (2001), no.~5, 601--615 (English).

\bibitem{DeLoeraRambauSantos:2010}
Jes{\'u}s~A. De~Loera, J{\"o}rg Rambau, and Francisco Santos,
  \emph{Triangulations}, Algorithms and Computation in Mathematics, vol.~25,
  Springer-Verlag, Berlin, 2010, Structures for algorithms and applications.

\bibitem{DevelinSturmfels:2004}
Mike Develin and Bernd Sturmfels, \emph{Tropical convexity}, Doc. Math.
  \textbf{9} (2004), 1--27 (electronic), erratum \textit{ibid.}, pp.~205--206.

\bibitem{DressWenzel:1992}
Andreas W.~M. {Dress} and Walter {Wenzel}, \emph{{Valuated matroids.}}, {Adv.
  Math.} \textbf{93} (1992), no.~2, 214--250 (English).

\bibitem{FinkRincon:2015}
Alex Fink and Felipe Rinc{\'o}n, \emph{Stiefel tropical linear spaces}, J.
  Combin. Theory Ser. A \textbf{135} (2015), 291--331.

\bibitem{Friedmann:2011}
Oliver Friedmann, \emph{Exponential lower bounds for solving infinitary payoff
  games and linear programs}, Ph.D. thesis, University of Munich, 2011.

\bibitem{Fukuda:1982}
Komei Fukuda, \emph{Oriented matroid programming}, ProQuest LLC, Ann Arbor, MI,
  1982, Thesis (Ph.D.)--University of Waterloo (Canada).

\bibitem{DMV:polymake}
Ewgenij Gawrilow and Michael Joswig, \emph{{\texttt{polymake}}\xspace: a
  framework for analyzing convex polytopes}, Polytopes---combinatorics and
  computation (Oberwolfach, 1997), DMV Sem., vol.~29, Birk\-h\"au\-ser, Basel,
  2000, pp.~43--73.

\bibitem{GelfandKapranovZelevinsky:1994}
I.~M. Gelfand, M.~M. Kapranov, and A.~V. Zelevinsky, \emph{Discriminants,
  resultants, and multidimensional determinants}, Mathematics: Theory \&
  Applications, Birkh\"auser Boston, Inc., Boston, MA, 1994.

\bibitem{Handbook+DCG:2004}
Jacob~E. {Goodman} and Joseph {O'Rourke} (eds.), \emph{{Handbook of discrete
  and computational geometry. 2nd ed.}}, 2nd ed. ed., Boca Raton, FL: Chapman
  \& Hall/CRC, 2004 (English).

\bibitem{Gruenbaum:1972}
B.~Gr{\"u}nbaum and Conference~Board of~the Mathematical~Sciences,
  \emph{Arrangements and spreads}, Regional conference series in mathematics,
  Conference Board of the Mathematical Sciences, 1972.

\bibitem{GKK:1988}
V.~A. Gurvich, A.~V. Karzanov, and L.~G. Khachiyan, \emph{Cyclic games and
  finding minimax mean cycles in digraphs}, Zh. Vychisl. Mat. i Mat. Fiz.
  \textbf{28} (1988), no.~9, 1407--1417, 1439.

\bibitem{HerrmannJoswigSpeyer:2014}
Sven Herrmann, Michael Joswig, and David~E. Speyer, \emph{Dressians, tropical
  {G}rassmannians, and their rays}, Forum Math. \textbf{26} (2014), no.~6,
  1853--1881.

\bibitem{Horn1}
Silke Horn, \emph{A topological representation theorem for tropical oriented
  matroids}, J. Combin. Theory Ser. A \textbf{142} (2016), 77--112.

\bibitem{Joswig:2020}
Michael Joswig, \emph{Essentials of tropical combinatorics}, in preparation.

\bibitem{Kapranov:1993}
M.~M. {Kapranov}, \emph{{Chow quotients of Grassmannians. I.}}, {I. M. Gelfand
  seminar. Part 2: Papers of the Gelfand seminar in functional analysis held at
  Moscow University, Russia, September 1993}, Providence, RI: American
  Mathematical Society, 1993, pp.~29--110 (English).

\bibitem{Loho:2016}
Georg {Loho}, \emph{{Abstract tropical linear programming}}, {Electron. J.
  Comb.} \textbf{27} (2020), no.~2, research paper p2.51, 68 (English).

\bibitem{LohoSmith:2020}
Georg Loho and Ben Smith, \emph{Matching fields and lattice points of
  simplices}, Advances in Mathematics \textbf{370} (2020), 107232.

\bibitem{LukowskiParisiWilliams:2020}
Tomasz Lukowski, Matteo Parisi, and Lauren Williams, \emph{{The positive
  tropical Grassmannian, the hypersimplex, and the $m=2$ amplituhedron}}, 2020,
  preprint \href{https://arxiv.org/abs/2002.06164}{arXiv:2002.06164}.

\bibitem{Massouros:1991}
Ch.~G. Massouros, \emph{Constructions of hyperfields}, Math. Balkanica (N.S.)
  \textbf{5} (1991), no.~3, 250--257.

\bibitem{Mnev:1988}
N.~E. {Mnev}, \emph{{The universality theorems on the classification problem of
  configuration varieties and convex polytopes varieties.}}, {Topology and
  geometry, Rohlin Semin. 1984-1986, Lect. Notes Math. 1346, 527-543 (1988).},
  1988.

\bibitem{Murota:2003}
Kazuo {Murota}, \emph{{Discrete convex analysis.}}, Philadelphia, PA: SIAM
  Society for Industrial and Applied Mathematics, 2003 (English).

\bibitem{OhYoo:2011}
Suho Oh and Hwanchul Yoo, \emph{Triangulations of
  {$\Delta_{n-1}\times\Delta_{d-1}$} and tropical oriented matroids}, 23rd
  {I}nternational {C}onference on {F}ormal {P}ower {S}eries and {A}lgebraic
  {C}ombinatorics ({FPSAC} 2011), Discrete Math. Theor. Comput. Sci. Proc., AO,
  Assoc. Discrete Math. Theor. Comput. Sci., Nancy, 2011, pp.~717--728.

\bibitem{OhYoo-ME:2013}
\bysame, \emph{{Triangulations of $\Delta_{n-1}\times\Delta_{d-1}$ and Matching
  Ensembles}}, 2013, preprint
  \href{https://arxiv.org/abs/1311.6772}{arXiv:1311.6772}.

\bibitem{Postnikov:2009}
Alexander Postnikov, \emph{Permutohedra, associahedra, and beyond}, Int. Math.
  Res. Not. IMRN (2009), no.~6, 1026--1106.

\bibitem{Rincon:2013}
Felipe Rinc\'{o}n, \emph{Local tropical linear spaces}, Discrete Comput. Geom.
  \textbf{50} (2013), no.~3, 700--713.

\bibitem{Santos:2005}
Francisco Santos, \emph{The {C}ayley trick and triangulations of products of
  simplices}, Integer points in polyhedra---geometry, number theory, algebra,
  optimization, Contemp. Math., vol. 374, Amer. Math. Soc., Providence, RI,
  2005, pp.~151--177.

\bibitem{Schrijver:1979}
A.~Schrijver, \emph{Matroids and linking systems}, J. Combin. Theory Ser. B
  \textbf{26} (1979), no.~3, 349--369.

\bibitem{SpeyerSturmfels:2004}
David {Speyer} and Bernd {Sturmfels}, \emph{{The tropical Grassmannian.}},
  {Adv. Geom.} \textbf{4} (2004), no.~3, 389--411 (English).

\bibitem{SpeyerWilliams:2020}
David Speyer and Lauren Williams, \emph{{The positive Dressian equals the
  positive tropical Grassmannian}}, 2020, preprint
  \href{https://arxiv.org/abs/2003.10231}{arXiv:2003.10231}.

\bibitem{Speyer:2008}
David~E. Speyer, \emph{Tropical linear spaces}, SIAM Journal on Discrete
  Mathematics \textbf{22} (2008), no.~4, 1527--1558.

\bibitem{SturmfelsZelevinsky:1993}
Bernd Sturmfels and Andrei Zelevinsky, \emph{Maximal minors and their leading
  terms}, Adv. Math. \textbf{98} (1993), no.~1, 65--112.

\bibitem{Viro:2010}
Oleg Viro, \emph{{Hyperfields for Tropical Geometry I. {H}yperfields and
  dequantization}}, 2010, preprint
  \href{https://arxiv.org/abs/1006.3034}{arXiv:1006.3034}.

\end{thebibliography}

\end{document}